\documentclass{amsart}

\usepackage{graphicx}
\usepackage{snaptodo}
\usepackage{wrapfig}
\usepackage{pifont}
\usepackage[dvipsnames]{xcolor}

\usepackage{mathpazo}

\usepackage{amsmath,amssymb,amsfonts,amsthm,mathrsfs}
\usepackage{dsfont} \usepackage[shortlabels]{enumitem}
\usepackage[dvipsnames]{xcolor}
\usepackage{tikz}
\usepackage{tikz-cd} \usepackage{tkz-graph} \usepackage{stmaryrd} \usepackage{pgfplots}
\usepackage{setspace} \usepackage{mathtools}
\usepackage{array}
\usepackage{microtype}
\usepackage[british]{babel}
\usepackage{hyperref}
\usepackage{geometry} \usepackage[h]{esvect} \usepackage{subcaption}
\usepackage{multirow}
\usepackage{cleveref} 

\usepackage{algorithm}
\usepackage{algpseudocode}

\usepackage{longtable}
\usepackage{booktabs}

\numberwithin{equation}{section} 

\newtheorem{thm}{Theorem}[section]

\newtheorem{lemma}[thm]{Lemma}
\newtheorem{prop}[thm]{Proposition}

\newtheorem{defn}[thm]{Definition}
\newtheorem{rmk}[thm]{Remark}
\newtheorem{ex}[thm]{Example}

\newcommand{\IR}{\mathbb R}
\newcommand{\IZ}{\mathbb Z}

\newcommand{\IF}{\mathbb F}
\newcommand{\MC}[3][\ast]{\mathrm{MC}_{#1, #2}(#3)} \newcommand{\EMC}[3][\ast]{\mathrm{EMC}_{#1, #2}(#3)} \newcommand{\MH}[3][\ast]{\mathrm{MH}_{#1, #2}(#3)}
\newcommand{\EMH}[3][\ast]{\mathrm{EMH}_{#1, #2}(#3)}
\newcommand{\MHcent}[1]{C_{\mathrm{MH}}(#1)}
\newcommand{\EMHcent}[1]{C_{\mathrm{EMH}}(#1)}
\DeclareMathOperator{\im}{im}
\DeclareMathOperator{\compl}{compl}

\DeclareMathOperator{\V}{V}
\DeclareMathOperator{\E}{E}
\DeclareMathOperator{\h}{H}

\newcommand{\tick}{\Pisymbol{pzd}{51}}
\newcommand{\cross}{\Pisymbol{pzd}{55}}

\DeclarePairedDelimiter\abs{\lvert}{\rvert}

\renewcommand{\epsilon}{\ensuremath\varepsilon}

\renewcommand{\phi}{\ensuremath{\varphi}}

\hypersetup{
    colorlinks = true,
    urlcolor   = RoyalBlue,
    linkcolor  = RoyalBlue,
    citecolor  = RoyalBlue
}

\pgfplotsset{
    compat = 1.18,
} 
\usepackage{amsaddr}

\title{A Centrality Measure Using Magnitude Homology}

\author{Nadja Häusermann and Bastian Rieck}
\address{AIDOS Lab, University of Fribourg, Switzerland}
\email{\url{nadja.haeusermann@unifr.ch}, \url{bastian.grossenbacher@unifr.ch}}

\usepackage[giveninits, sortcites, maxbibnames = 99]{biblatex}
\usepackage{csquotes}
\begin{document}

\begin{abstract}
  The \emph{magnitude} of a metric space constitutes an expressive
  invariant that subsumes numerous different geometrical-topological
  invariants. Building on recent advances in \emph{magnitude homology},
  i.e., a bigraded homology theory that recovers the magnitude, we develop a novel local measure of the
  \emph{centrality} or \emph{importance} of nodes in a graph.
Our measure is inspired by the concept of relative homology as it considers the change in magnitude homology when removing a vertex. We show that our proposed measure satisfies several properties a centrality measure is reasonably expected to respect and demonstrate that we introduce a new perspective on centrality by comparing to several established centrality measures.
\end{abstract}

\maketitle

\section{Introduction}

The \emph{magnitude} of a metric space refers to a novel invariant that captures geometrical-topological aspects at different scales.
Originally developed to assess the diversity of a metric space, i.e., its effective number of points~\cite{leinster2013magnitude_metric}, numerous aspects of magnitude have been formalized, linking it to other entropy measures~\cite{leinster2021entropy_diversity}, for instance.
In parallel, magnitude has also received an extension to graphs in the form of \emph{magnitude homology}, enabling wider applicability~\cite{hepworth2017categorifying} and discriminative power.\footnote{In the sense that there are graphs that have the same magnitude but different magnitude homology~\cite{gu2018magnitude_homology_morse}.
}
Magnitude homology expands on the `expressivity' of ordinary homology by being able, among other things, to capture the diameter of `holes' in metric spaces~\cite{kaneta2021mh_metric_spaces_order_complexes} while also reflecting the uniqueness of geodesics, i.e., shortest paths in a graph~\cite{gomi2025magnitude_homology_geodesic}.
It may thus be considered a more geometric version of homology.
A more recent variation of magnitude homology, the Eulerian magnitude homology, consists of smaller homology groups while still capturing essential structures of the graph, as it has been linked to counting substructures~\cite{giusti2024eulerianMH_subgraph_structure}.
Motivated by recent uses of magnitude in machine learning and data science applications~\cite{magnitude_evaluation_diversity_latent_space, magnitude_and_generalisations_in_NN,bunch2021weighting_vectors_boundary_detection}, we develop a novel \emph{centrality measure} based on magnitude homology and its Eulerian version.

Centrality is a crucial concept in graph and network analysis, referring to the overall importance of a vertex or edge in a graph.
Leveraging counts of shortest paths and subgraph structures is often exploited by existing centrality measures~\cite{freeman1977centrality_betweenness,estrada2005subgraph_centrality}, hence we investigate how the magnitude homology, which is linked to these concepts, can be a useful tool in developing a novel type of vertex centrality. 
While there are numerous centrality measures described in the literature~(the reader is referred to \textcite{saxena2020centrality_measures_complex_networks} for a recent survey), our magnitude-based centrality measure is inspired by \emph{relative homology}. 
That is, we capture changes in magnitude homology resulting from \emph{removing} a given vertex, with the additional grading of magnitude homology being used as a natural locality parameter.
Next to proving that our graded family of centrality measures, which we refer to as \emph{(Eulerian) magnitude homology centrality}, is well-defined and respects several axioms that are desired in a centrality measure, we conduct extensive experiments to demonstrate that our measure serves as a useful local characterization of a graph.

We provide the required background about magnitude homology and Eulerian magnitude homology in \Cref{sec:background}. Then, in \Cref{sec:magnitude_homology_centrality} we propose our novel node centrality measure based on the (Eulerian) magnitude homology and show that it is local. To motivate the use of the magnitude homology centrality, we study in \Cref{sec:axioms} what axioms of centrality measures are satisfied. Finally, in \Cref{sec:results} we demonstrate the proposed centrality measure on concrete examples and compare to established centrality measures.
\section{Background: Magnitude Homology and Eulerian Magnitude Homology}
\label{sec:background}

In order to using metric space magnitude on graphs, we equip graphs with their shortest-path distance. Due to the shortest-path metric being integer-valued, the magnitude of graphs can be viewed as an integer power series \cite{leinster2019magnitude_graph}. 
\textcite{hepworth2017categorifying} introduced the magnitude homology of graphs, a bigraded homology theory which recovers the magnitude by taking the graded Euler characteristic.  
In this section, we recall the necessary definitions and properties for defining and studying the magnitude homology centrality. We further give the definition of a variant of the magnitude homology introduced in \textcite{giusti2024eulerianMH_subgraph_structure}, the Eulerian magnitude homology. Since we define a homology, we assume a basic familiarity with the algebraic notion of homology through a chain complex. For a formal treatment of these concepts, we refer to \textcite{weibel2013homological_algebra}.

We first establish preliminary notions for~(Eulerian) magnitude
homology. Subsequently, we denote an undirected graph by $G = (V, E)$,
with $V$ being the \emph{vertices} and $E \subseteq V \times V$ being
the edges, we denote an edge between vertices $u,v \in V$ by $(u,v)$. In particular, we assume that $G$ is
\emph{simple}, i.e., it has no self-loops and no multiple edges.
For any integer $k\geq 0$ we denote by $[k]$ the set $\{1, \dots ,k\}$.

\subsection{Magnitude Chains and Eulerian Magnitude Chains}

Since the magnitude is based on a metric space, we have to equip our graphs with a metric. For this work, we will be using the shortest-path metric, as this will produce a bigraded vector space in homology indexed by $\IZ_{\geq 0}$, allowing for a simpler computation.

\begin{defn}[Shortest-path distance in a graph]
    For any graph $G=(V,E)$ we denote by $d\colon V\times V \to
    [0,\infty]$ the shortest-path distance. That is, $d(u,v) = \infty$ for $u,v \in V$ if $u$ and $v$ are not in the same connected component and otherwise the distance $d(u,v)$ is the number of edges in a shortest path from $u$ to $v$. 
\end{defn}

The magnitude homology has also been extended to general metric spaces that do not necessarily arise from graphs with the shortest-path distance \cite{leinster2021magnitude_homology_enriched_cat_metric_space}, hence it is also possible to use other graph distances, but we leave this extension for future work.

Chains in the magnitude homology group correspond to paths of a given length that have to pass through a sequence of vertices, called landmarks, where they have to use a shortest path between two consecutive landmarks.
\begin{defn} [Length of a tuple in a graph]
    The length $\ell$ of a tuple $(x_0,\dots,x_k) \in V^{k+1}$ of vertices of a graph $G$ is the sum of the shortest-path distance between consecutive vertices
\begin{equation}
      \ell(x_0,\dots,x_k) = \sum_{i=0}^{k-1} d(x_i,x_{i+1}).
    \end{equation}
Note that for $k=0$, the length of a tuple $(x_0)$ is $\ell(x_0) = 0$.
\end{defn}

\begin{defn}[Magnitude chains]
    The magnitude chain group of a graph $G$ in bidegree $(k,l)$ for $k, l \geq 0$ is the $\mathbb{F}_2$-vector space $\MC[k]{l}{G}$ generated by the elements $(x_0, \dots, x_k) \in V^{k+1} $ such that $\forall i \in \{0, \dots k\} \; x_i \neq x_{i+1}$, i.e., consecutive vertices in the tuple are distinct, and $\ell(x_0, \dots, x_k) = l$.
\end{defn}

As opposed to only the \emph{consecutive} vertices being distinct in the definition above of the magnitude chain group, there exists also a version where \emph{all} vertices in a generator have to be distinct, leading to the Eulerian magnitude homology \cite{giusti2024eulerianMH_subgraph_structure}.

\begin{defn}[Eulerian magnitude chains]
    The Eulerian magnitude chain group of a graph $G$ in bidegree $(k,l)$ for $k, l \geq 0$ is the $\mathbb{F}_2$-vector space $\EMC[k]{l}{G}$ generated by the elements $(x_0, \dots, x_k) \in V^{k+1}$ such that $\forall\, 0 \leq i \neq j \leq k, x_i \neq x_j$ and $\ell(x_0, \dots, x_k) = l.$
\end{defn}

Note that the~(Eulerian) magnitude chain group is usually defined
    as the free abelian group on the same generators, meaning that they are introduced as above but with $\IZ$ coefficients instead of $\IF_2$ coefficients. In our paper, following standard practice in computational topology, we
    use $\mathbb{F}_2$ coefficients for easier computation; we will also
    make use of the vector space properties of $\mathbb{F}_2$ to show
    the \emph{locality} of our magnitude homology centrality measure.

\subsection{Magnitude Homology and Eulerian Magnitude Homology}

To obtain a \emph{chain complex}, we need differentials between the magnitude chain groups, relating the chains in different degrees with each other. \textcite{hepworth2017categorifying} show that the maps defined below satisfy $\partial \circ \partial = 0$, hence they do define a differential on the~(Eulerian) magnitude chain groups. 

\begin{defn}[Differential on magnitude chains]
    For a graph $G$ and integers $l\geq 0, k \geq 1$ we define the differential
    \begin{equation*}
    \partial \colon \MC[k]{l}{G} \to \MC[k-1]{l}{G}
  \end{equation*}
  by the alternating sum $\partial = \sum_{i=1}^{k-1}
  \partial_i$, where $\partial_i \colon \MC[k]{l}{G} \to \MC[k-1]{l}{G}$
  is defined on the generators by
  \begin{equation}
    \label{eq definition differential summands}
    \partial_i (x_0,\dotsc,x_k) =
      \begin{cases}
          (x_0,\dots,\widehat{x_i},\dots,x_k) \quad & \text{if } \ell(x_0,\dotsc,\widehat{x_i},\dotsc,x_k) = l \\
          0 &\text{otherwise}
      \end{cases}
    \end{equation}
and then linearly extended to the whole space $\MC[k]{l}{G}$.
Following standard conventions, $\widehat{x_i}$ indicates that
    vertex $x_i$ is \emph{excluded} from a tuple.
\end{defn}

\begin{defn}[Magnitude chain complex]
    The magnitude chain complex $\MC{\ast}{G}$ of a graph $G$ is the direct sum of chain complexes, i.e.,
    \begin{equation*}
      \bigoplus_{l\geq 0} \MC{l}{G}.
    \end{equation*}
\end{defn}

We only consider one component of the magnitude chain complex of
    a graph $G$ at a time. Hence, for simplicity, we will also refer to
    the chain complex $\MC{l}{G}$ for a fixed $l \geq 0$ as \emph{the}
    magnitude chain complex of $G$.

We are now ready to define the (Eulerian) magnitude homology. The general idea of a homology is to consider only the chains that are \emph{cycles}, i.e., the chains being zero under the differential, where we quotient out the \emph{boundaries}, i.e., the chains that are in the \emph{image} of the differential. The idea behind taking the quotient of the boundaries is to view two cycles as the same if they differ only by a boundary. Considering this relation gives interesting insights about the structure of the underlying object being studied. In the case of singular homology of a topological space, this construction allows us to count the connected components and holes of the space. The interpretation of the magnitude homology is still being explored. For instance, we know that we recover the number of edges of the underlying graph with the magnitude homology and that there are connections to the girth and uniqueness of geodesics of the graph \cite{asao2024girth_mh_diagonality, gomi2025magnitude_homology_geodesic}. More generally, the magnitude homology of metric spaces is linked to the convexity and diameter of holes of the underlying metric space \cite{leinster2021magnitude_homology_enriched_cat_metric_space, kaneta2021mh_metric_spaces_order_complexes}.

\begin{defn}[Magnitude homology]
    The magnitude homology $\MH{\ast}{G}$ of a graph $G$ is the
    bigraded vector space defined by taking the homology
\begin{equation*}
      \MH[k]{l}{G} = \h_k(\MC{l}{G}).
    \end{equation*}
    That is, the quotient \[
    \MH[k]{l}{G} = \ker(\partial: \MC[k]{l}{G} \to \MC[k-1]{l}{G}) / \im(\partial: \MC[k+1]{l}{G} \to \MC[k]{l}{G}). \]
\end{defn}

For small degrees, one can easily compute the magnitude homology groups, as shown by the following result, which will help us interpret the proposed centrality measures.
\begin{prop}[{\textcite[Proposition 9]{hepworth2017categorifying}}]
\label{prop:mh_degree_00_11}
    For any graph $G$, the magnitude homology satisfies:
        \begin{enumerate}[i)]
        \item $\MH[0]{0}{G} \cong \langle \V(G) \rangle$;
        \item $\MH[1]{1}{G} \cong \langle \vv{\mathrm{E}}(G) \rangle$;
    \end{enumerate}
where $\vv{\E}$ denotes the set of oriented edges of $G$, that is, for every undirected edge $(u,v)$ of $G$, the set $\vv{\E}$ contains both directions $(u,v), (v,u) \in \vv{\E}$. 
It immediately follows that
    \begin{enumerate}[i)]
        \item $\dim(\MH[0]{0}{G}) = \# \V(G)$;
        \item $\dim(\MH[1]{1}{G}) = 2 \cdot \# \E(G)$.
    \end{enumerate}
\end{prop}

The non-zero cycles in magnitude homology include magnitude chains that alternate between adjacent vertices in the graph. For certain classes of graphs and in certain bidegrees, these are in fact the only types of non-zero magnitude homology classes \cite{giusti2024eulerianMH_subgraph_structure}. 
Based on this observation, the Eulerian magnitude homology was introduced in order to remove those cycles that do not provide new information \cite{giusti2024eulerianMH_subgraph_structure}. 

\begin{defn}[Eulerian magnitude homology]
    We can define a differential on the Eulerian magnitude chain groups
    analogously to the differential on the magnitude chain groups, 
    obtaining the Eulerian magnitude chain complex $\bigoplus_{l\geq
    0} \EMC{l}{G}$ with the same convention as for the standard case.
    The Eulerian magnitude homology $\EMH{\ast}{G}$ of a graph
    $G$ is the bigraded vector space defined by the homology groups
    \begin{equation*}
      \EMH[k]{l}{G} = \h_k(\EMC{l}{G}).
    \end{equation*}
\end{defn}

For small $k$, the standard and Eulerian magnitude homology agree, as the following result shows.

\begin{lemma}
    For $k= 0, 1$, any graph $G$, and $l \geq 0$, magnitude homology and Eulerian magnitude homology are isomorphic, i.e., $\MH[k]{l}{G}
    \cong \EMH[k]{l}{G}$.
\end{lemma}
\begin{proof}
    The statement follows immediately from the definition: For $k=0,1$, the magnitude and Eulerian magnitude chains consist of tuples with one or two vertices. In this case,
    there is no difference whether we require only consecutive vertices or all pairs of vertices to be distinct. Hence, the Eulerian and standard magnitude chain groups are isomorphic and, as a consequence, so are their resulting homologies.
\end{proof}

\begin{ex}
    We provide a brief worked example of magnitude chains and their behaviour in (Eulerian) magnitude homology, the relevant graph and magnitude chains with length $l=3$ and $k=2$ are demonstrated in \Cref{fig:ex_mc}. 
    
    Consider first the magnitude chain $(0,2,3)$ in \Cref{fig:ex_mc1}. To understand its differential, we need to study what happens if we remove the middle vertex, $2$. Removing it leaves the chain $(0,3)$, with the dotted green path still being a shortest path from $0$ to $3$. Hence, the length of the tuple does not change and $\partial(0,2,3) = (0,3) \neq 0$. In other words, $(0,2,3)$ is \emph{not} a cycle in (Eulerian) magnitude homology. If we imagine the dotted grey edge $(1,3)$ to be part of the graph, the situation would change. In that case, without the condition that the chain must pass through the landmark $2$, the shortest path between $0$ and $3$ only has length $2$, rendering the differential of $(0,2,3)$ to be $0$. Hence, if all vertices in the tuple are relevant in the sense that there is a shorter path if we do not enforce to pass through all landmarks, then the generator is a cycle.
    
    Let us now consider the chain $(1,2,4)$, as highlighted in \Cref{fig:ex_mc2}. It is a cycle in (Eulerian) magnitude homology since the red edge $(1,4)$ ensures that removing the vertex $2$ strictly decreases the length of the tuple. In standard magnitude homology, $(1,2,4)$ is the image under the differential of $(1,2,1,4)$, hence the chain $(1,2,4)$ will be trivial in homology. Meanwhile, in Eulerian magnitude homology, no repeating vertices are allowed, therefore $(1,2,1,4)$ is not an Eulerian magnitude chain and one can check that there is in fact no Eulerian magnitude chain in with image $(1,2,4)$ under the differential. It follows that $(1,2,4)$ is non-trivial in Eulerian magnitude homology.
\begin{figure}[h!]
    \begin{subfigure}[t]{0.45\textwidth}
        \includegraphics[width=0.8\textwidth]{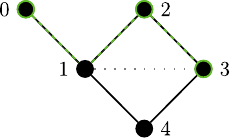}
        \caption{The chain $(0,2,3)$ does not have vanishing differential in the (Eulerian) magnitude chain complex.}
        \label{fig:ex_mc1}
    \end{subfigure} \quad
    \begin{subfigure}[t]{0.45\textwidth}
        \includegraphics[width=0.8\textwidth]{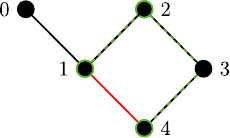}
        \caption{The chain $(1,2,4)$ is a cycle of the (Eulerian) magnitude chains that is trivial in standard magnitude homology, but not in Eulerian magnitude homology.}
        \label{fig:ex_mc2}
    \end{subfigure}
    \caption{Examples of (Eulerian) magnitude chains in bidegree $k=2$, $l=3$ with the vertices in the chain highlighted in green and one underlying walk demonstrated with the dotted green edges.}
    \label{fig:ex_mc}
\end{figure}
\end{ex}

\subsection{Induced Maps}

We study later how the inclusion of a subgraph into the whole graph changes the (Eulerian) magnitude homology. To do this, we need to briefly introduce how maps of graphs induce maps in homology. Formally speaking, this describes how the magnitude homology defines a functor from the category of graphs to the category of vector spaces. This section follows the work by \textcite[Section 3]{hepworth2017categorifying}.

First, we fix the type of maps between graphs that can induce maps in homology. Specifically, the right notion turns out to be a map between graphs that sends an edge to an edge or collapses it to a point. 
\begin{defn}[Map of graphs]
    For graphs $G, H$, we define a map of graphs $f\colon G \to~H$ to be a map $f\colon \V(G) \to \V(H)$ on the vertex sets such that
\begin{equation*}
    \forall{x, y} \in \E(G) \colon (f(x), f(y)) \in \E(H) \textrm{ or } f(x) = f(y).
    \end{equation*}
\end{defn}
This definition has an equivalent characterization in terms of the shortest-path distance of the graphs. Namely, a map between graphs is non-increasing in terms of the shortest-path distance.
\begin{rmk}
\label{rmk maps of graph decrease distance}
    Equivalently, we can define a map of graphs $f\colon G \to H$ to be a map on the vertex sets $f\colon \V(G) \to \V(H)$ such that \[
    \forall x,y \in \V(G) \qquad d_H(f(x),f(y)) \leq d_G(x, y).
    \]
\end{rmk}

A map of graphs induces a map on the magnitude chains by mapping each vertex in the chain to its image if that preserves the length of the tuple, otherwise we send it to zero.

\begin{defn}[Induced chain map]
    If $f\colon G \to H$ is a map of graphs, the induced chain map $f_\#\colon  \MC{\ast}{G} \to \MC{\ast}{H}$ is defined on generators by \[
    f_\#(x_0,\dots,x_k) = \begin{cases}
        (f(x_0),\dots,f(x_k)) \quad &\text{if } \ell(f(x_0),\dots,f(x_k)) = \ell(x_0,\dots,x_k) \\
        0 &\text{else}.
    \end{cases}
    \]
\end{defn}

The induced map $f_\#$ defined above is indeed a chain map, that is, the relation $f_\# \circ \partial = \partial \circ f_\#$ is satisfied in every bidegree and hence we obtain a map in homology.

\begin{defn}[Induced map in homology]
    Let $f\colon G\to H$ be a map of graphs. The induced map in homology is the map \[
    f_\ast\colon \MH{\ast}{G} \to \MH{\ast}{H}
    \]
    induced by the chain map $f_\#$.
\end{defn}

\begin{rmk}
\label{rmk induced map linear}
    Note that we define $f_\#$ on the basis of the magnitude chain vector space, hence it extends to a linear map between vector spaces and thus the induced map in homology is also a linear map between the magnitude homology vector spaces.
\end{rmk}

In the setting of Eulerian magnitude homology, the maps of graphs that induce chain maps are not allowed to be contractions \cite{caputi2025emh_diagonality_injective_words_regular_path_homology}. This does not pose any problems for our purposes are we are only using the induced maps of inclusions, which also satisfy this stronger property.

\section{A Centrality Measure Based on Magnitude Homology}
\label{sec:magnitude_homology_centrality}

In this section, we propose the (Eulerian) magnitude homology centrality, a novel centrality measure on vertices of undirected graphs. We will show that it is a local measure, where the size of the neighbourhood that influences the centrality measure is determined by the grading obtained from the (Eulerian) magnitude homology. 

\subsection{Definition of the (Eulerian) Magnitude Homology Centrality}
We start by introducing some necessary notations.

\begin{figure}
    \centering
    \includegraphics[width=0.9\linewidth]{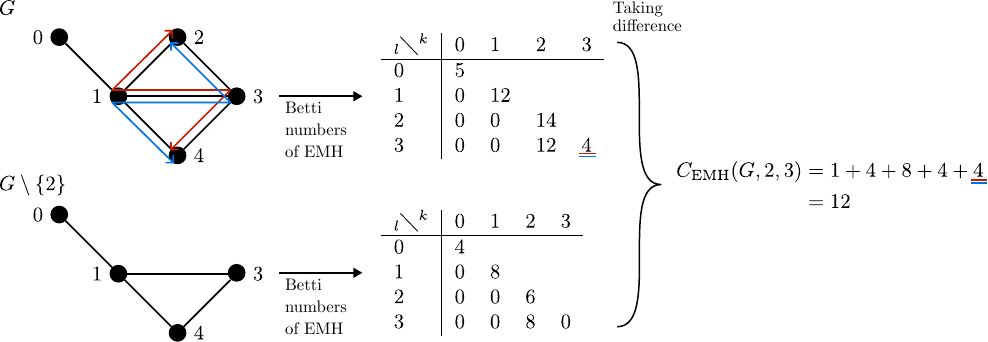}
    \caption{Computation of the  Eulerian magnitude homology centrality. To compute the centrality of the vertex $2 \in G$ we consider the Betti numbers of $G$ and its subgraph $G\setminus \{2\}$. For each bidegree up to a given $l$ ($l=3$ in the example), we compute the difference of the Betti numbers and add them up to obtain the Eulerian magnitude homology centrality $\EMHcent{G,2,3}$. By removing the vertex $2$, the blue and red generators in both directions that determine the dimension $\EMH[3]{3}{G}$ get destroyed, resulting in the summand $4$ of the Eulerian magnitude homology centrality.}
    \label{fig:overview_cent}
\end{figure}

\begin{defn}
    Let $G = (V, E)$ be a graph and $A \subset V$ a subset. We denote by $G \setminus A$ the graph induced by the subset $V \setminus A$ on $G$.
\end{defn} 

\begin{defn}[Betti numbers]
    We denote the Betti numbers of the (Eulerian) magnitude homology in bidegree $(k,l)$ of a graph $G$ by $\beta_{k, l}(G) = \dim(\MH[k]{l}{G})$ and $\beta_{k, l}^E(G) = \dim(\EMH[k]{l}{G})$.
\end{defn}

With these notions in place, we now define our (Eulerian) magnitude homology centrality. It is a centrality measure that tracks how the removal of a vertex affects the magnitude homology of the graph. This concept is inspired by the local homology of a topological space, which is defined via its relative homology. Notice, however, that a completely analogous definition to the relative homology is not possible due to the fact that the (Eulerian) magnitude chain complex of a subgraph is not a subcomplex of the (Eulerian) magnitude chain complex of the whole graph because the shortest path distance can change by going from the subgraph to the whole graph. In \Cref{fig:overview_cent}, we show how the Eulerian magnitude homology centrality is computed, the standard magnitude homology centrality is completely analogous.

\begin{defn}[(Eulerian) magnitude homology centrality]
    The (Eulerian) magnitude homology centrality $C_{\mathrm{MH}}(G, v, l)$ and $C_{\mathrm{EMH}}(G, v, l)$, respectively, of a vertex $v$ from a graph $G = (V,E)$ and in degree $l$ is defined to be the sum of the differences in Betti numbers \begin{align}
        \label{eq:defn_MHC}
         \MHcent{G, v, l} &= \sum_{l' = 0}^{l} \sum_{k' = 0}^{l'} (\beta_{k', l'}(G) - \beta_{k', l'}(G \setminus \{v\})) \\
         \label{eq:def_EMHC}
         \EMHcent{G, v, l} &= \sum_{l' = 0}^{l} \sum_{k' = 0}^{l'} (\beta_{k', l'}^E(G) - \beta_{k', l'}^E(G \setminus \{v\}))
\end{align}   
\end{defn}

Note that the summand $k'$ in the definition essentially varies over $\IZ_{\geq 0}$, but since (Eulerian) magnitude homology only has non-zero dimension in bidegrees $(k,l)$ with $k \leq l$ \cite[Proposition 10]{hepworth2017categorifying}, we already restrict to the finitely many possibly non-zero summands in the definition.

As an immediate consequence of \Cref{prop:mh_degree_00_11}, we can compute the value of the (Eulerian) magnitude homology centrality for $l=1$. We obtain that this case recovers the degree of the vertex, which is being used as a centrality measure \cite{saxena2020centrality_measures_complex_networks}.
\begin{prop}
\label{prop:EMHC_l=0}
    For a graph $G=(V,E)$ and any vertex $v \in V$, the (Eulerian) magnitude centrality $C_{\mathrm{MH}}(G,v,1) = C_{\mathrm{EMH}}(G,v,1) = 1+2\cdot\deg_G(v)$.
\end{prop}
\begin{proof}
    By \Cref{prop:mh_degree_00_11}, we know that $\beta_{0, 0}(G) - \beta_{0, 0}(G \setminus \{v\}) = 1$ because one vertex is removed and $\beta_{1, 1}(G) - \beta_{1, 1}(G \setminus \{v\}) = 2\cdot\deg(v)$.
\end{proof}

We note that the summands in \eqref{eq:defn_MHC} and \eqref{eq:def_EMHC} can be negative, that is, it can happen that $\beta_{k,l}(G\setminus \{v\}) > \beta_{k,l}(G)$. For example, in the Icosahedral graph $\mathcal{I}$ for any vertex $v \in \mathcal{I}$, the difference in Betti numbers is $\beta_{2,3}(\mathcal{I}) - \beta_{2,3}(\mathcal{I} \setminus \{v\}) = 0 - 10$ and for the Eulerian case $\beta_{2,3}^E(\mathcal{I}) - \beta_{2,3}^E(\mathcal{I} \setminus \{v\}) = 0 - 30$. We leave it to future work to investigate when this phenomenon occurs. 

\begin{ex}
    Let us consider a small example and calculate the (Eulerian) magnitude centrality for the four-cycle $C_4$ graph, see \Cref{fig:four_cycle}. 
    \begin{figure}[h]
    \begin{subfigure}[t]{0.45\textwidth}
        \centering
        \begin{tikzpicture}[scale = 1]
    \GraphInit[vstyle=Classic]
    \SetVertexMath
    \SetGraphUnit{1}
    \tikzset{VertexStyle/.append style={minimum size=6pt}}
    \Vertex[Lpos=180]{0}
    \SO[Lpos=180](0){1}
    \EA(0){3}
    \SO(3){2}

    \Edges(0, 1, 2, 3, 0)
    
\end{tikzpicture}         \caption{Four-cycle.}
        \label{fig:four_cycle}
    \end{subfigure}
    \begin{subfigure}[t]{0.45\textwidth}
        \centering
        \begin{tikzpicture}[scale = 1]
    \GraphInit[vstyle=Classic]
    \SetVertexMath
    \SetGraphUnit{1}
    \tikzset{VertexStyle/.append style={minimum size=6pt}}
    \Vertex[Lpos=180]{1}
    \EA(1){2}
    \NO(2){3}

    \Edge(1)(2)
    \Edge(2)(3)
    
\end{tikzpicture}         \caption{Four-cycle minus the vertex $v=0$.}
        \label{fig:C4_minus_0}  
    \end{subfigure}
    \caption{The four-cycle $C_4$ for which we compute the (Eulerian) magnitude homology centrality and its subgraph, the path graph on three vertices, when deleting a vertex.}
    \end{figure}
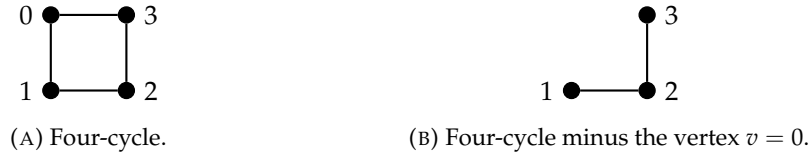
     
    The Betti numbers of the magnitude homology and Eulerian magnitude homology up to $k=l=4$ are given in \Cref{tab:betti_MH_C4} and \Cref{tab:betti_EMH_C4}, the Betti numbers of standard magnitude homology are known~\cite{gu2018magnitude_homology_morse}, we did compute all Betti numbers using computer calculations~(cf.\  \Cref{sec:computation_EMH}). 
    \begin{table}[h!]
    \parbox{.45\linewidth}{
    \centering
    \begin{tabular}{l|lllll}
    $_l \diagdown ^k$ & $0$ & $1$ & $2$ & $3$ & $4$ \\ \hline
    $0$ & $4$ &  &  &  &  \\
    $1$ & $0$ & $8$ &  &  &  \\
    $2$ & $0$ & $0$ & $12$ &  &  \\
    $3$ & $0$ & $0$ & $0$ & $16$ &  \\
    $4$ & $0$ & $0$ & $0$ & $0$ & $20$ 
    \end{tabular}
    \caption{Betti numbers of the magnitude homology of $C_4$.}
    \label{tab:betti_MH_C4}
    }
    \hfill\parbox{.45\linewidth}{
    \centering
    \begin{tabular}{l|lllll}
    $_l \diagdown ^k$ & $0$ & $1$ & $2$ & $3$ & $4$ \\ \hline
    $0$ & $4$ &  &  &  &  \\
    $1$ & $0$ & $8$ &  &  &  \\
    $2$ & $0$ & $0$ & $4$ &  &  \\
    $3$ & $0$ & $0$ & $8$ & $0$ &  \\
    $4$ & $0$ & $0$ & $0$ & $8$ & $0$ 
    \end{tabular}
    \caption{Betti numbers of the Eulerian magnitude homology of $C_4$.}
    \label{tab:betti_EMH_C4}   
    }
\end{table} 
Removing any vertex from the four cycle graph leaves the path graph $P_3$ on three vertices. For our example, say we choose to compute $\MHcent{C_4, 0, l}$ for $l \leq 4$, hence we remove the vertex $v=0$, as can be seen in \Cref{fig:C4_minus_0}. 
The Betti numbers of the magnitude and Eulerian magnitude homologies of $C_4 \setminus \{0\}$ can be seen in \Cref{tab:Betti_MH_three_path} and \Cref{tab:Betti_EMH_three_path}.
\begin{table}[h!]
\parbox{.475\linewidth}{
    \centering
    \begin{tabular}{l|lllll}
    $_l \diagdown ^k$ & $0$ & $1$ & $2$ & $3$ & $4$ \\ \hline
    $0$ & $3$ &  &  &  &  \\
    $1$ & $0$ & $4$ &  &  &  \\
    $2$ & $0$ & $0$ & $4$ &  &  \\
    $3$ & $0$ & $0$ & $0$ & $4$ &  \\
    $4$ & $0$ & $0$ & $0$ & $0$ & $4$ 
    \end{tabular}
    \caption{Betti numbers of the magnitude homology of $C_4 \setminus \{0\}$.}
    \label{tab:Betti_MH_three_path}}
\hfill
\parbox{.475\linewidth}{
    \centering
    \begin{tabular}{l|lllll}
    $_l \diagdown ^k$ & $0$ & $1$ & $2$ & $3$ & $4$ \\ \hline
    $0$ & $3$ &  &  &  &  \\
    $1$ & $0$ & $4$ &  &  &  \\
    $2$ & $0$ & $0$ & $0$ &  &  \\
    $3$ & $0$ & $0$ & $4$ & $0$ &  \\
    $4$ & $0$ & $0$ & $0$ & $0$ & $0$ 
    \end{tabular}
    \caption{Betti numbers of the Eulerian magnitude homology of $C_4 \setminus \{0\}$.}
    \label{tab:Betti_EMH_three_path}}
\end{table} 
The differences of the Betti numbers between the full graph $C_4$ and $C_4 \setminus \{0\}$ are shown in \Cref{tab:diff_Betti_MH_cycle} and \Cref{tab:diff_Betti_EMH_cycle}. Summing up the entries in the rows up to a specific $l$ gives the (Eulerian) magnitude homology centrality for that degree. As we have shown in \Cref{prop:EMHC_l=0}, the difference in dimension $(k,l) = (0,0)$ is $1$ and in bidegree $(k,l) = (1,1)$ the difference is precisely twice the degree of the vertex we remove. 
\begin{table}[h!]
\parbox{.475\linewidth}{
    \centering
    \begin{tabular}{l|lllll}
    $_l \diagdown ^k$ & $0$ & $1$ & $2$ & $3$ & $4$ \\ \hline
    $0$ & $1$ &  &  &  &  \\
    $1$ & $0$ & $4$ &  &  &  \\
    $2$ & $0$ & $0$ & $8$ &  &  \\
    $3$ & $0$ & $0$ & $0$ & $12$ &  \\
    $4$ & $0$ & $0$ & $0$ & $0$ & $16$ 
    \end{tabular}
    \caption{Differences in Betti numbers of the magnitude homology of $C_4$ and $C_4 \setminus \{0\}$.}
    \label{tab:diff_Betti_MH_cycle}}
\hfill
\parbox{.475\linewidth}{
    \centering
    \begin{tabular}{l|lllll}
    $_l \diagdown ^k$ & $0$ & $1$ & $2$ & $3$ & $4$ \\ \hline
    $0$ & $1$ &  &  &  &  \\
    $1$ & $0$ & $4$ &  &  &  \\
    $2$ & $0$ & $0$ & $4$ &  &  \\
    $3$ & $0$ & $0$ & $4$ & $0$ &  \\
    $4$ & $0$ & $0$ & $0$ & $8$ & $0$ 
    \end{tabular}
    \caption{Differences in Betti numbers of the Eulerian magnitude homology of $C_4$ and $P_3$.}
    \label{tab:diff_Betti_EMH_cycle}}
\end{table}

Consider now the Eulerian magnitude homology spaces. Let us look at $l = 2$. The only non-trivial spaces can occur for $k=1,2$. The generators of $\EMC[1]{2}{C_4}$ are \begin{align*}
    (0,2), (1,3), (2, 0), (3, 1)
\end{align*} and the generators of $\EMC[2]{2}{C_4}$ are \begin{align*}
    &(0,1,2), (0,3,2), (1,2,3), (1,0,3) \\
    &(2,1,0), (2,3,0), (3,2,1), (3,0,1).
\end{align*}

Since every generator of $\EMC[1]{k}{C_4}$ lies already in the image of the boundary, the corresponding homology space is trivial, $\EMH[1]{2}{C_4} = 0$. The kernel $\ker(\partial\colon \EMC[2]{2}{C_4} \to \EMC[1]{2}{C_4})$ is generated by the chains \[
(0,1,2) + (0,3,2), (1,2,3) + (1,0,3), (2,1,0)+(2,3,0), (3,2,1)+(3,0,1),
\]
so we get $\dim(\EMH[2]{2}{C_4}) = 4$. \textcite{giusti2024eulerianMH_subgraph_structure} show a relation between the dimension of Eulerian magnitude homology and the count of certain substructures in a graph. Specifically,
let $C_3$ and $C_4$ be the three- and four-cycles and $F_4$
be the four cycle with one diagonal edge added, and let $c(G,H)$ denote the number of induced subgraphs of $G$ isomorphic to $H$. Then it was shown that the Eulerian magnitude homlogy with $\IZ$ coefficients satisfies \[
\dim(\EMH[2]{2}{G}) \leq 6\cdot c(G,C_3) + 4\cdot c(G, C_4) + 2\cdot c(G, F_4)
\] for any graph $G$. In our example, only $c(C_4, C_4) = 1$ is adding a non-zero term to the right hand side and we have in fact an equality. Removing one vertex destroys the $C_4$ subgraph and the above result implies that $\dim(\EMH[2]{2}{C_4 \setminus \{0\}}) = 0$. Indeed, when removing the vertex, we are left with the generators for $\EMC[1]{2}{C_4 \setminus \{0\}}$ \[
(1,3), (3,1),
\] and for $\EMC[2]{2}{C_4 \setminus \{0\}}$ we have \[
(1,2,3), (3,2,1)
\] as generators. In this case, the kernel of the differential $\partial\colon \EMC[2]{2}{C_4 \setminus \{0\}} \to \EMC[1]{2}{C_4 \setminus \{0\}}$ is trivial and the image is again the whole group $\EMC[1]{2}{C_4 \setminus \{0\}}$. The difference in Betti numbers in bidegree $(2,2)$ hence captures that removing the vertex $0$ destroys the cyclic substructure.

For the standard magnitude, the cyclic substructure also gets destroyed and that is reflected in the differences of the Betti numbers, but additionally there are terms of chains going back and forth on adjacent vertices, so in particular the two edges incident to the vertex $0$ that are removed will keep inducing chains that appear in the difference.
\end{ex}

\subsection{Computation of the (Eulerian) Magnitude Homology Centrality}
\label{sec:computation_EMH}
We briefly discuss the computation of the (Eulerian) magnitude homology centrality. We have implemented a naive algorithm computing the (Eulerian) magnitude homology and our proposed centrality measures in Python, we discuss the code in \Cref{app:code} and it is available on GitHub.\footnote{The code is available at \href{https://github.com/aidos-lab/mh_cent}{https://github.com/aidos-lab/mh\_cent}.} 
Using this implementation, we compute the examples presented in this paper and perform several experiments in \Cref{sec:axioms,sec:results}. With our implementation, already the complexity for computing all generators of the magnitude chain complex in bidegree $(k,l)$ for a graph with $n$ vertices that start with a given vertex and correspond to some partition of $l$ with $k$ summands is of order $\mathcal{O}(n^k)$.
Indeed, the problem of computing the Eulerian magnitude homology is being studied and there has been shown a relation to the subgraph isomorphism problem \cite{giusti2024eulerianMH_subgraph_structure}, which is known to be NP-complete in the general case \cite{wegener2005complexity_theory}. 
\textcite{menara2024computing_eulerian_mh} study the problem of computing the diagonal Eulerian magnitude homology Betti numbers and propose a breadth first algorithm that still runs exponential in $k$ in the worst case, but has sub-exponential or polynomial complexity in many real-world cases. This is due to the fact that their algorithm grows exponentially in the diameter of the graph and many real world graphs exhibit small diameters. In our application of the Eulerian magnitude homology centrality, we control the diameter of the neighbourhood through the parameter $l$ as explained in the subsequent section, indicating potential to further improve on our algorithm for computing (Eulerian) magnitude homology centrality.

\subsection{Locality of the (Eulerian) Magnitude Homology Centrality}
We show that the above defined centrality measure is in fact a local measure. That is, we do not need the information of the whole graph to compute the magnitude homology centrality of degree $l$ of a specific vertex, but it suffices to know the structure of the $l$-hop neighbourhood. In this section, we formalise this statement and prove it.

\begin{defn}[$l$-hop neighbourhood]
    Let $G = (V,E)$ be a graph and $v \in V$ a vertex. The neighbourhood subgraph $G_v^l$ of the vertex $v$ and of radius $l$ is the subgraph induced by the $l$-hop neighbourhood \[
    \mathcal{N}_l^G(v) = \{u \in V| d(u,v) \leq l\}.
    \] 
\end{defn}

\begin{defn}[Local version of (Eulerian) magnitude homology centrality]
    The local version of (Eulerian) magnitude homology centrality of a vertex $v$ from the graph $G$ in degree $l$ is defined to be the magnitude homology centrality
    \begin{align}
        \label{eq locrelMH centrality}
         \MHcent{G_v^l, v, l} &= \sum_{l' = 0}^{l} \sum_{k' = 0}^{l'} (\beta_{k', l'}(G_v^{l}) - \beta_{k', l'}(G_v^{l} \setminus \{v\})) \\
         \EMHcent{G_v^l, v, l } &= \sum_{l' = 0}^{l} \sum_{k' = 0}^{l'} (\beta_{k', l'}^E(G_v^{l}) - \beta_{k', l'}^E(G_v^{l} \setminus \{v\}))
    \end{align}   
\end{defn}
We now prove preliminary results in order to show that the local magnitude homology centrality is equal to the global magnitude homology centrality in the standard case in \Cref{prop:locality}. For the Eulerian version the proof is the same. 
\begin{prop}
\label{prop centrality inclusion map relation}
    Denote by $f\colon G \setminus \{v\} \to G$ the inclusion map. The magnitude homology centrality can be expressed as \[
    \MHcent{G, v, l} = \sum_{l'=0}^l \sum_{k'=0}^{l'} (\dim(\compl(\im(f_\star))) - \dim(\ker(f_\star))),
    \] where $f_\ast$ is to be understood as the induced map in homology in bidegree corresponding to the sum indices. Furthermore, $\compl(\im(f_\star))$ denotes the complement of the subspace $\im(f_\star) \subset \MH{\star}{G}$.
\end{prop}
\begin{proof}
    First, note that the inclusion $f$ is a map of graphs, hence the induced map in homology is well defined and by Remark \ref{rmk induced map linear} it is a linear map between the magnitude homology vector spaces.
    We consider fixed $l, k \geq 0$ and show that \[
    \dim(\MH[k]{l}{G}) - \dim(\MH[k]{l}{G \setminus \{v\}} = \dim(\compl(\im(f_\star))) - \dim(\ker(f_\star)),\]
    where we consider $f_\star$ to be the induced map in homology in bidegree $(k,l)$.
    Note that the homology vector space splits into a direct sum $\MH[k]{l}{G} = \im(f_\star) \oplus \compl(\im(f_\star))$ and hence $\dim(\MH[k]{l}{G}) = \dim(\im(f_\star)) + \dim(\compl(\im(f_\star)))$. By the rank-nullity theorem we furthermore get the relation $\dim(\MH[k]{l}{G \setminus \{v\}} = \dim(\im(f_\star)) + \dim(\ker(f_\star))$. Putting these together we obtain \[
        \dim(\MH[k]{l}{G}) - \dim(\MH[k]{l}{G\setminus \{v\} } = \dim(\compl(\im(f_\star))) - \dim(\ker(f_\star)).
    \]
\end{proof}

The next definitions formalises the concept of all vertices that are involved in a magnitude chain. The idea being that if we can understand that a chain only depends on the vertices in some neighbourhood, then changing the graph outside of that neighbourhood will not affect the chain. Care has to be taken since the shortest path between two vertices can be non-unique (in case of cycles in the graph) and we want to include all of these possible paths that are defined by a magnitude chain.

\begin{defn}[Underlying walk]
    An underlying walk of a generator $\bar{x} = (x_0, \dots, x_k) \in \MC[k]{l}{G}$ is a walk that passes all vertices $x_0, \dots, x_k$ in that order and for every $i \in [k]$ takes a shortest path between $x_{i-1}$ and $x_i$. We denote the collection of all underlying walks of the generator $\bar{x}$ by $\mathcal{W}(\bar{x})$. Note that by definition, the length of an underlying walk is equal to $l$.
\end{defn}

\begin{defn}[Support of a magnitude chain]
    Let $\sum_j \bar{x_j} \in \MC[k]{l}{G}$ be any element. The set of vertices $\{x \in W | W \in \cup_j \mathcal{W}(\bar{x_j})\}$ is called the support of $\sum_j \bar{x_j}$ and denoted by $\mathcal{S}(\sum_j \bar{x_j})$.
\end{defn}
In their analysis of the cycles of the diagonal Eulerian magnitude homology, \textcite{giusti2024eulerianMH_subgraph_structure} introduce the following definition. 
\begin{defn}[Minimal cycle]
    Let $\alpha = \sum_{j \in J} \bar{x_j} \in \MC[k]{l}{G}$ be a cycle, i.e., the differential $\partial(\alpha)$ vanishes. We say that $\alpha$ is a \emph{minimal cycle} if there is no subset $J' \subsetneq J$ such that $\partial(\sum_{j \in J'}\bar{x_j}) = 0$.
\end{defn}
A special case of the following lemma, where $|J| = 2$, was already mentioned in \textcite[Example 3.2]{giusti2024eulerianMH_subgraph_structure}.

\begin{lemma}
    \label{lemma endpoints cycle}
    Let $\alpha = \sum_{j \in J} \bar{x_j} \in \MC[k]{l}{G}$ be a minimal cycle. Then for all $j_1, j_2 \in J$, the start and endpoints of $\bar{x_{j_1}}$ and $\bar{x_{j_2}}$ agree.
\end{lemma}
\begin{proof}
    Let $\alpha  = \sum_{j \in J} \bar{x_j} \in \MC[k]{l}{G}$ be a minimal cycle and let $\cup_{n \in N} J_n = J$ be a partition of $J$ such that all chains corresponding to some $J_n$ have the same start and end points. By the definition of the differential, $\partial(\alpha) = \sum_{j \in J} \partial(\bar{x_j}) = \sum_{j \in J} \sum_{i=1}^{k-1} \partial_i(\bar{x_j})$. The maps $\partial_i$ potentially remove a vertex from the tuple at position $i \in \{1,\dots, k-1\}$, in particular the start and end points remain. For the differential of $\alpha$ to vanish, this implies that already $\sum_{j \in J_n} \partial(\bar{x_j})$ for any $n \in N$ must vanish because summands with different start and endpoints cannot cancel each other out in $\MC[k]{l}{G}$. Since $\alpha$ is minimal, it follows that $\abs{N} = 1$ and hence all summands of $\alpha$ must already have the same start and endpoints.
\end{proof}

\begin{rmk}
    \label{rmk support of boundary}
    Consider a magnitude chain $\bar{x} = (x_0, \dots x_{k+1}) \in \MC[k+1]{l}{G}$ and any element $ \alpha	\neq 0 \in \MC[k]{l}{G}$ such that $\partial(\bar{x}) = \alpha$. Because $\alpha \neq 0$, there is at least one index $i \in [k]$ such that $\partial_i(\bar{x}) = (x_0, \dots, \hat{x_i}, \dots, x_{k+1})$. In particular, the length $\ell(x_0, \dots, \hat{x_i}, \dots, x_{k+1}) = l$ and hence every underlying walk of $\bar{x}$ is still an underlying walk of $(x_0, \dots, \hat{x_i}, \dots, x_{k+1})$. The other direction does not necessarily hold, there might be underlying walks of $(x_0, \dots, \hat{x_i}, \dots, x_{k+1})$ that are not underlying walks of $\bar{x}$. In other words, if the removal of a vertex does not change the length, then that vertex was in a way not essential for determining the underlying walks of the tuple. 
\end{rmk}

\begin{lemma}
\label{lemma:locality_support}
Let $v \in G$ be a vertex and $\alpha = \sum_j \bar{x_j} \in \MC[k]{l}{G}$ a minimal cycle such that $v$ is in the support of $\sum_j \bar{x_j}$. Then $\mathcal{S}(\sum_j\bar{x_j}) \subset \mathcal{N}_l(v)$.
\end{lemma}
\begin{proof}
    By Lemma \ref{lemma endpoints cycle}, every summand in the cycle has the same start and end point, so we simply write $x_0$ and $x_k$ for them respectively.
    Let $w \in \mathcal{S}(\sum_j\bar{x_j})$ be any vertex in the support. 
    We need to show that $d(v, w) \leq l$. 
    Since both $v$ and $w$ are in the support of $\alpha$, there are walks $W_1$ and $W_2$ from $x_0$ to $x_k$, both of length $l$. By concatenating the two walks, we obtain a closed walk of length $2l$. 
    Since the vertices $v$ and $w$ lie on this closed walk, we obtain two paths from $v$ to $w$ by taking the path through $x_0$ and $x_1$, respectively. The sum of the lengths of these walks from $v$ to $w$ is $2l$, hence at least one of them has length less than or equal to $l$. It follows that the shortest-path distance $d(v,w) \leq l$.
\end{proof}

\begin{rmk}
\label{rmk:locality_generator}
    From the proof of the above lemma, we obtain the following special case: Let $v \in G$ be any vertex and $\bar{x} = (x_0, \dots, x_k) \in \MC[k]{l}{G}$ be any generator such that $x_i = v$ for some $i \in [k]$. Then, for any underlying walk $W \in \mathcal{W}(\bar{x})$, any vertex $w \in W$ is contained in the neighbourhood $\mathcal{N}_l(v)$.
\end{rmk}

\begin{prop}
\label{prop:locality}
    The (Eulerian) magnitude homology centrality is a local centrality, that is, $\MHcent{G, l, v} =  \MHcent{G_v^l, l, v}$ and $\EMHcent{G, l, v} = \EMHcent{G_v^l, v, l}$ for any vertex $v \in G$ and for any $l \geq 0$.
\end{prop}
\begin{proof}
    Let $f\colon G \setminus \{v\} \to G$ and $f^\ell \colon G_v^\ell \setminus \{v\} \to G_v^\ell$ denote the inclusions. By Proposition \ref{prop centrality inclusion map relation}, it is enough to show that for any $k, l \geq 0$, the dimensions satisfy \begin{align}
        \dim(\compl(\im(f_\star))) &= \dim(\compl(\im(f_\star^\ell))) \label{eq local proof dim compl im}\\ \dim(\ker(f_\star)) &= \dim(\ker(f_\star^\ell)) \label{eq local proof dim ker}.
    \end{align}
    First, we show Equation \eqref{eq local proof dim compl im}. There are two types of basis elements of $\compl(\im(f_\star))$, which are: \begin{itemize}
        \item Homology classes of generators that contain the vertex $v$ in one of their summands.
        \item Homology classes that do not contain $v$ in any of their summands but for which at least one of the respective $k$-tuples does have a different length in $G\setminus \{v\}$. The only possibility for this to happen is if removing the vertex $v$ increases the length of that summand, see \Cref{rmk maps of graph decrease distance}.
    \end{itemize}
    We first note that if the basis consists of non-minimal cycles, by splitting them up into the minimal cycles, we obtain a generating set and can extract a basis of minimal cycles from this. Hence, consider an element of the first type. The homology class can be represented as a minimal cycle $\alpha = \sum_j (x_0^j, \dots, x_k^j)$ of $(k,l)$-chains, where there exist some $j'$ and $i'$ such that $x_{i'}^{j'} = v$. By \Cref{lemma:locality_support} and particularly \Cref{rmk:locality_generator}, the support of $\alpha$ lies in the $l$-hop neighbourhood $\mathcal{N}^G_l(v)$ of $G$.
    
    Let $\sum_j \bar{x_j}$ represent a homology class of $\MH[k]{l}{G}$ of the second type. That is, $v$ does not appear as a vertex in any generator of $\sum_j \bar{x}_j$ and there exists an index $j'$ such that $\ell_{G \setminus \{v\}}(\bar{x}_{j'}) > l$. For this to happen, the vertex $v$ must lie on an underlying walk of $\bar{x}_{j'}$ in $G$ because this is the only vertex that differs between $G$ and $G \setminus \{v\}$. Then again by Lemma \ref{lemma:locality_support}, it follows that the support of $\sum_j (x_0^j, \dots, x_k^j)$ lies in the neighbourhood $\mathcal{N}^G_l(v)$.
    
    We have shown that every element of $\compl(\im(f_\star))$ already has support in $G \setminus \{v\}$, which proves that $\compl(\im(f_\star)) \subset \compl(\im(f_\star^l))$ (note that this is a slight abuse of notation; by writing this, we mean that any generator $\alpha = \sum_j (x_0^j, \dots, x_k^j)$ representing a homology class in $\compl(\im(f_\star))$ also represents a homology class in $\compl(\im(f_\star^l))$). On the other hand, every element $\sum_j \bar{x_j} \in \compl(\im(f_\star^l))$ lies also in $\compl(\im(f_\star))$ if there is no vertex of $G \setminus \mathcal{N}_l(v)$ in the support of $\sum_j \bar{x_j}$. But by the same reasoning as above, the vertex $v$ is in the support of $\sum_j \bar{x_j}$ and hence the support $\mathcal{S}(\sum_j \bar{x_j}) \subset \mathcal{N}_l(v)$, showing that also $ \compl(\im(f_\star^l)) \subset \compl(\im(f_\star))$
    We deduce that $\dim(\compl(\im(f_\star))) = \dim(\compl(\im(f_\star^l)))$.
    
    Second, we prove that Equation \eqref{eq local proof dim ker} holds. The type of basis elements of the kernel $\dim(\ker(f_\star))$ are: \begin{itemize}
        \item Homology classes of chains for which at least one of their summands has a different length if we add vertex $v$. Note that the only possibility for this to happen is if the length decreases if we add the vertex $v$ to the graph, see \Cref{rmk maps of graph decrease distance}.
        \item Homology classes for which none of the lengths of their summands change if we add vertex $v$ and that have non-zero homology in $G \setminus \{v\}$ but zero homology in $G$, that is, their image under $f_\#$ lies in $\im(\partial\colon \MC[k+1]{l}{G} \to \MC[k]{l}{G})$.
    \end{itemize}
By the same argument as above, consider a homology class of the first type represented by the minimal cycle $\sum_j \bar{x_j} \in \MC[k]{l}{G \setminus \{v\}}$. For any summand $\bar{x_{j'}}$ for which the length decreases if viewed in $G$, that is, $\ell_{G \setminus \{v\}}(\bar{x_{j'}}) = l$ and $\ell_G(\bar{x_{j'}}) < l$ the vertex $v$ must lie on an underlying path in $G$ because that is the only way to decrease the length. By using Lemma \ref{lemma:locality_support} we obtain that the support of the representative  $\mathcal{S}(\sum_j \bar{x_j}) \subset \mathcal{N}^G_l(v)$.
    
    For the second type consider a representative $\alpha = \sum_j \bar{x_j} \in \MC[k]{l}{G}$. From the property that $\alpha \notin \im(\partial\colon \MC[k+1]{l}{G\setminus \{v\}} \to \MC[k]{l}{G \setminus \{v\}}$ but $f_\#(\alpha) \in \im(\partial\colon \MC[k+1]{l}{G} \to \MC[k+1]{l}{G})$ it follows that there exists some homology class $\beta \in \MC[k+1]{l}{G}$ such that $\partial(\beta) = f_\#(\alpha)$. Furthermore, there is a summand of $\beta$ for which $v$ lies on an underlying walk and which does not vanish under the differential, otherwise $\alpha$ would also be a boundary in $\MC[k]{l}{G \setminus \{v\}}$. From Remark \ref{rmk support of boundary} we can conclude that the vertex $v$ is also on an underlying walk of $\alpha$ in $G$ and with Lemma \ref{lemma:locality_support} the support of $\alpha$ is in the $l$-hop neighbourhood $\mathcal{N}^G_l(v)$.
    
    As for the first equation, we have shown that an element in $\ker(f_\star)$ has support in $\mathcal{N}_l(v)$ and therefore $\ker(f_\star) \subset \ker(f_\star^l)$. Conversely, any element in $\ker(f_\star^l)$ can also be viewed as an element of $\ker(f_\star)$ if its support is in $\mathcal{N}^G_l(v)$ even when viewed as a tuple in $G$. But this holds because $v$ must be in the support by the above argument, hence also $\ker(f_\star^l) \subset \ker(f_\star)$. We can conclude that Equation \eqref{eq local proof dim ker} holds as well.
\end{proof}

The previous proof also give us some idea about what the magnitude centrality captures. We described that the elements in $\compl(\im(f_\star))$ are of two possible types. Elements of the first type on the diagonal occur for example if all vertices in the tuple are in a clique. In that case, removing any vertex always results in a strictly shorter length because the neighbours are connected and since we are on the diagonal those generators are not boundaries, hence they define non-zero elements in the (Eulerian) magnitude homology. The latter type is linked to counting on how many shortest paths the vertex $v$ lies. Since the same type of chains also appears as one of the types appearing in $\ker(f_\star)$, some of them cancel each other out, but there is a shift in the degree ensuring that at least in the highest degree of $l$ those second type of chains are counted in the (Eulerian) magnitude homology centrality. Hence our centrality measure does give some weight to vertices that are parts of dense clusters but also vertices that are on shortest paths of other pairs of vertices. We consider a more in-depth analysis of what is counted in the (Eulerian) magnitude homology centrality to be the topic of future work, though, focusing here on its definition, computation, and overall properties instead.

\section{Axioms of Centrality Measures}
\label{sec:axioms}

\begin{wrapfigure}[15]{r}{0.25\textwidth}
    \centering
    \includegraphics[width=\linewidth]{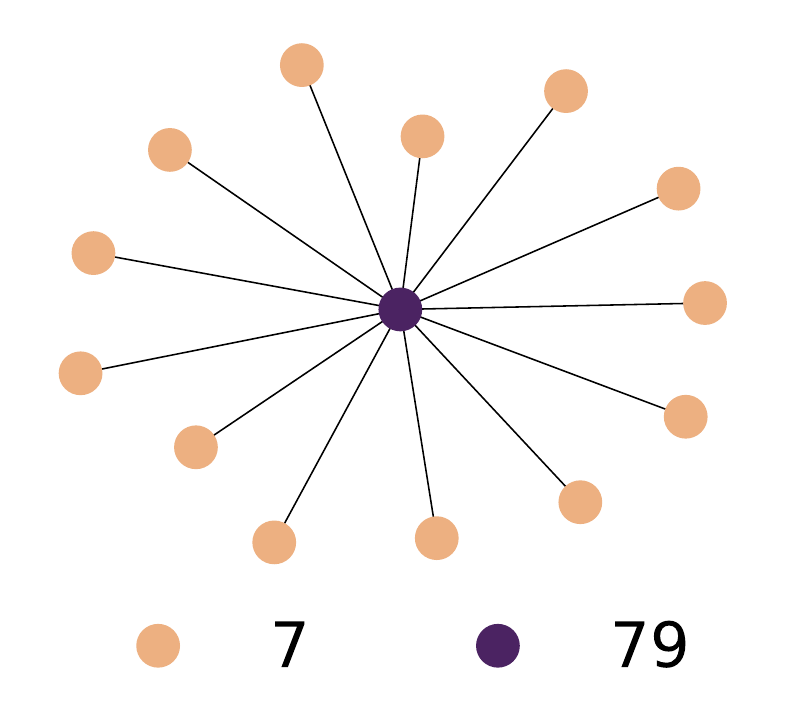}
    \caption{The magnitude homology centrality for $l=3$ of the star graph.}
    \label{fig:star_example}
\end{wrapfigure}
In this section, we motivate that the (Eulerian) magnitude centrality from the previous section indeed can be viewed as a measure of centrality in a graph by validating that it behaves how we would expect a centrality measure to behave.
This is not a straightforward task, as the notion of centrality in a graph is not strictly defined. Intuitively, the idea of centrality is to describe the importance of nodes in a network given a specific application. 
There is a general agreement that on the example of the star graph, the vertex with the highest centrality measure is the centre of the star \cite{sabidussi1966centrality_index,freeman1978centrality_conceptual_clarification}. \Cref{fig:star_example} shows that the magnitude homology centrality respects that intuition.
Apart from this concrete example, there have been several approaches to define a set of characteristics that a centrality measure should fulfil based on the intuition. Notably, \textcite{sabidussi1966centrality_index} was the first to try and define a set of axioms for centrality measures by considering operations of adding and switching edges in a graph. For a recent survey over the different axioms that have been proposed, we refer to \textcite{meshcheryakova2024comparative_analysis_centrality}. Furthermore, \textcite{boldi2014axioms_centrality} propose a set of three axioms for centrality measures on directed graphs with the approach of using strongly connected vertex transitive graphs, namely complete graphs and cycles, to state two of their three proposed axioms. Due to their simplicity, we discuss our proposed centrality measure on the set of axioms introduced by \textcite{boldi2014axioms_centrality}, translated to undirected graphs. An overview of the results from this section can be found in \Cref{tab:axiom_overview}, where we also include several established centrality measures as a reference. We introduce these baseline centrality measures in \Cref{sec:baseline_centrality_measures} when directly comparing them to our newly proposed (Eulerian) magnitude centralities. Note that for the (Eulerain) magnitude homology centrality, we restrict ourselves to $l \in \{2,3\}$, as we have only used these degrees in our experimental section. 

Formally, a centrality measure on the vertices of a graph $G= (V,E)$ is a function $c\colon V \to \IR_{\geq 0}$. Generally, it is also required that $c$ is invariant under isomorphism \cite{sabidussi1966centrality_index,boldi2014axioms_centrality}. That is, for a graph isomorphism $f\colon G \to H$ and any vertex $x \in G$ there is equality $c(x)=x(f(x))$, meaning that the centrality only depends on the structure of the graph. This property also implies that vertices in the same orbit of the automorphism group of the graph have the same centrality value. In particular for vertex transitive graphs, that is graphs where there is only one vertex orbit under the automorphism group, any centrality measure assigns the same value to all vertices in the graph. The (Eulerian) magnitude homology centrality satisfies the isomorphism invariance, since the (Eulerian) magnitude homology itself is an isomorphism invariant.
\begin{table}[]
    
    \begin{center}
        \begin{tabular}{lcccc}
        Centrality Measure & Size & Density & Endpoint Increase & Top Node \\
        \midrule
        $C_\mathrm{MH}$ with $l\in \{2,3\}$ & \textcolor{Green}{weak} & \textcolor{Green}{\tick} & \textcolor{red}{\cross} & \textcolor{Green}{\textbf{?}} \\
        $C_\mathrm{EMH}$ with $l\in \{2,3\}$ & \textcolor{Green}{weak} & \textcolor{Green}{\textbf{?}} & \textcolor{red}{\cross} & \textcolor{Green}{\textbf{?}} \\
        \midrule
        Degree & \textcolor{Green}{weak} & \textcolor{Green}{\tick} & \textcolor{Green}{\tick} & \textcolor{Green}{\tick} \\
        Betweenness & \textcolor{Green}{only \Cref{item:sizeaxiom_cycle}} & \textcolor{red}{\cross} & \textcolor{red}{\cross} & \textcolor{red}{\cross} \\
        Closeness & \textcolor{red}{\cross} & \textcolor{Green}{\tick} & \textcolor{Green}{\tick} & \textcolor{red}{\cross} \\
        PageRank & \textcolor{red}{\cross} & \textcolor{Green}{\tick} & \textcolor{Orange}{\textbf{?}} & \textcolor{Orange}{\textbf{?}}\\
        Subgraph & \textcolor{Orange}{\textbf{?}} & \textcolor{Orange}{\textbf{?}}& \textcolor{Orange}{\textbf{?}} & \textcolor{Orange}{\textbf{?}} \\
    
        \end{tabular}
    
    \end{center}

    \caption{Overview over what axioms are satisfied by different centrality measures. For degree, betweenness, closeness, PageRank, and subgraph centrality the results are taken from \textcite{meshcheryakova2024comparative_analysis_centrality,boldi2014axioms_centrality}. \textcolor{Green}{\tick}: axiom is satisfied; \textcolor{Green}{\textbf{?}}: axiom is conjectured to be true; \textcolor{Orange}{\textbf{?}}: unknown whether axiom is satisfied or not; \textcolor{red}{\cross}: axiom is not satisfied.}
    \label{tab:axiom_overview}
\end{table}

\subsection{The Size Axiom} The size axiom gauges how sensitive a centrality measure is to the size of the graph. The axiom is stated in terms of the disjoint union $K_n \sqcup C_p$ of the complete graph $K_n$ on $n$ vertices and the cycle graph $C_p$ on $p$ vertices. Each of the components is vertex transitive, so they have constant centrality measure.
\begin{defn}[Size axiom]
\label{def:size_axiom}
    A centrality measure satisfies the size axiom, if
    \begin{enumerate}[i)]
        \item \label{item:sizeaxiom_complete} for every $n \geq 0$, there exists $P_{n} \in \IZ_{\geq 0}$ such that for all $p \geq P_n$, the disjoint union $C_p \sqcup K_n$ satisfies that the centrality of a node in the cycle is strictly larger than the centrality of a node in the complete graph;
        \item \label{item:sizeaxiom_cycle} and conversely, if for any $p\geq 0$ there exists $N_p \in \IZ_{\geq 0}$ such that for all $n \geq N_p$ the disjoint union $C_p \sqcup K_n$ satisfies that the centrality of a node of the complete graph is strictly larger than the centrality of a node in the cycle.
    \end{enumerate}
    If a centrality measure only satisfies \ref{item:sizeaxiom_complete}, we say it satisfies the weak size axiom.
\end{defn}

The magnitude homology for for trees and for the complete graph with $\IZ$ coefficients are known and torsion free \cite{hepworth2017categorifying}. As a corollary of the universal coefficient theorem in homology \cite[Corollary 3A.6.]{hatcher2002algebraictopology}, the Betti numbers for integer coefficients are equal to the Betti numbers of the magnitude homology with $\IF_2$ coefficients:

\begin{align*}
    \beta_{k,l}(K_n) &= \begin{cases}
            n \cdot (n-1)^l \quad &\text{if } k=l \\
            0 &\text{else}.
        \end{cases} \\
    \beta_{k,l}(T_n) &= \begin{cases}
            n \quad &\text{if } k=l=0 \\
            2(n-1) \quad &\text{if } k=l>0 \\
            0 &\text{else},
        \end{cases} \\
\end{align*}
where $T_n$ denotes a tree on $n$ vertices.
By removing any vertex in the complete graph $K_n$, we obtain the complete graph $K_{n-1}$ and hence for any vertex $v \in K_n$, the magnitude homology centrality is \begin{align*}
    \MHcent{K_n, v, 2} &= 1 + 2(n-1) + (n-1)(3n+4) = 3n^2-5n+3 \\
    \MHcent{K_n, v, 3} &= \MHcent{K_n, v, 2} + (n-1)(n(n-1)-(n-2)^2) = 4n^3-12n^2+14n-5.
\end{align*}
Since we have shown that the magnitude homology centrality is local, we know that for any vertex $v \in C_p$, its centrality $C_{\mathrm{MH}}(C_p,v,l)$ eventually becomes constant for increasing $p$. In particular, the $l$-hop neighbourhood of $v$ for $p>2l+1$ is the path graph $P_{2l+1}$ and hence the centrality of $C_p$ is \begin{align*}
    \MHcent{C_p, v, 2} &= 1 + 2*2*2 = 9 \textrm{ for } p>5\\
    \MHcent{C_p, v, 3} &= 1 + 3*2*2 = 13 \textrm{ for } p>7.
\end{align*} 
We can conclude that for magnitude homology centrality for $l\in \{2,3\}$, the weak size axiom holds. The full size axiom cannot hold because magnitude homology is a \emph{local} centrality measure.

The Eulerian magnitude homology groups for complete graphs and trees are not known to the best of our knowledge. Hence, in the Eulerian case we fall back on our algorithm and compute the Eulerian magnitude homology for the complete graph and the cycle graph. 
Due to the locality we see that as in the standard case, the Eulerian magnitude centrality for the cycle satisfies \begin{align*}
    \EMHcent{C_p, v, 2} &= 5 \textrm{ for } p>5\\
    \EMHcent{C_p, v, 3} &= 17 \textrm{ for } p>7,
\end{align*} 
while the Eulerian magnitude homology of the complete graph $K_n$ seems to grow with increasing $n$ and in particular already \begin{align*}
    \EMHcent{K_4, v, 2} &= 25 \textrm{ for } p>5\\
    \EMHcent{K_4, v, 3} &= 49 \textrm{ for } p>7
\end{align*} 
is enough to show that Eulerian magnitude homology satisfies the weak size axiom.

\subsection{The Density Axiom} The density axiom aims to ensure that the centrality of a node with a dense neighbourhood is higher than that of a node in a sparse neighbourhood. The axiom is formulated in terms of the graph $D_{n,p}$ consisting of the complete graph $K_n$ and the cycle graph $C_p$ connected by one edge, denoted by $e=(v_C,v_K)$ where $v_K$ belongs to $K_n$ and $v_C$ belongs to $C_p$.
\begin{defn}[Density axiom]
    The density axiom is satisfied if for $n=p>3$, the centrality of node $v_K$ is strictly larger than the centrality of node $v_C$.
\end{defn}

For the (Eulerian) magnitude homology centrality, we need the Betti numbers of $D_{n,p}\setminus \{v_C\} = P_{p-1}\sqcup K_n$ and $D_{n,p}\setminus \{v_K\} = C_{p}\sqcup K_{n-1}$, so the condition $\MHcent{D_{n,p}, v_C, l} < \MHcent{D_{n,p}, v_K, l}$ is equivalent to 
\begin{equation}
    \label{eq:density_axiom}
\sum_{l' = 0}^{l} \sum_{k' = 0}^{l'} (\beta_{k', l'}(C_p) - \beta_{k', l'}(P_{p-1}) < \sum_{l' = 0}^{l} \sum_{k' = 0}^{l'} (\beta_{k', l'}(K_n) - \beta_{k', l'}(K_{n-1})
\end{equation}

by using the additivity of magnitude homology with respect to the disjoint union \cite[Proposition 17]{hepworth2017categorifying}.
If we fix $k,l\in \{0,\dots,3\}$, \textcite{gu2018magnitude_homology_morse} shows that for cycles $C_p$ with $p\geq 7$ the magnitude homology is torsion-free and \[
        \beta_{k,l}(C_p) = \begin{cases}
            p \quad &\text{if } k=l=0 \\
            2p &\text{if } 1 \leq k=l \leq 3 \\
            0  &\text{else for } 0\leq k,l \leq 3,
        \end{cases} 
\] so the same Betti numbers apply for $\IF_2$ coefficients. We can compute the summands of the left hand side of \eqref{eq:density_axiom} for $k,l \in \{0,\dots,3\}$ as \[ \beta_{k, l}(C_p) - \beta_{k, l}(P_{p-1}) = \begin{cases}
            p - (p-1) \quad &\text{if } k=l=0 \\
            2p - 2(p-2) &\text{if } 1 \leq k=l \leq 3 \\
            0  &\text{else for } 0\leq k,l \leq 3 \end{cases} = \begin{cases}
            1 \quad &\text{if } k=l=0 \\
            4 &\text{if } 1 \leq k=l \leq 3 \\
            0  &\text{else for } 0\leq k,l \leq 3. 
            \end{cases}
\]
For the right hand side of \eqref{eq:density_axiom}, we have \[
    \beta_{k, l}(K_n) - \beta_{k, l}(K_{n-1}) = \begin{cases}
            n - (n-1) \quad &\text{if } k=l=0 \\
            (n-1)(n(n-1)^{l-1}-(n-2)^l) &\text{if } 1 \leq k=l \leq 3 \\
            0  &\text{else for } 0\leq k,l \leq 3. \end{cases}
\] Particularly, for $0< k=l\leq 3$ and $n \geq 7$, the difference $\beta_{k, l}(K_n) - \beta_{k, l}(K_{n-1}) \geq 6$, so \eqref{eq:density_axiom} holds for $l\in \{2,3\}$ and $n\geq 7$. For $n<7$ we verify the axiom computationally and find that it holds. We conclude that for the magnitude homology centrality with $l\in \{2,3\}$ the density axiom holds.
As we do not have the analytic expressions in the Eulerian case, we verify the density axiom for $p=k \in \{4,\dots,30\}$ for the Eulerian magnitude homology centrality and find that for all the tested cases the axiom is satisfied, we thus conjecture that the density axiom is true for the Eulerian magnitude homology centrality with $l \in \{2,3\}$.

\subsection{The Endpoint Increase Axiom}
The last axiom stated in \textcite{boldi2014axioms_centrality} is the score-monotonicity axiom, which formalizes the idea, that adding an edge should increase the centrality of the incident vertices. We state the version of this axiom for undirected graphs \cite{meshcheryakova2024comparative_analysis_centrality}.
\begin{defn}[Endpoint increase axiom]
    The endpoint increase axiom is satisfied if for any graph $G$ with non-adjacent vertices $x,y \in V(G)$, adding the edge $(x,y)$ does not decrease the centrality of $x$ and $y$. 
\end{defn}
The intuition behind this axiom is that adding new connections to a node in a graph should not decrease its centrality.
We have empirically checked the satisfiability of this axiom for our proposed versions of Eulerian and non-Eulerian magnitude homology centrality with $l\in \{2,3\}$ on the $1253$ graphs with up to seven nodes named in the graph atlas \cite{read1998atlas_of_graphs}, as implemented by NetworkX \cite{hagberg2008networkx}. We have found that on these graphs, the endpoint increase axiom holds for the magnitude homology for both $l=2$ and $l=3$, while for the Eulerian version only for $l=2$ the axiom was satisfied. A counterexample where the axiom does not hold for the Eulerian magnitude homology centrality is the cycle graph $C_5$ when adding an edge between two vertices of distance $2$ apart. In the $5$-cycle all nodes $v \in C_5$ have centrality $C_{\mathrm{EMH}}(v, C_5, 3) = 27$ and adding the edge decreases the centrality of the incident vertices to $25$. After further testing using the `Netzschleuder' network catalogue \cite{peixeto2020netzschleuder}, we also found a counterexample for the magnitude homology centrality in the `Aconet' graph from the Internet Topology Zoo dataset \cite{knight2011internet_topology_zoo}. 
\begin{wrapfigure}[15]{r}{0.45\textwidth}
    \centering
    \begin{subfigure}{0.45\linewidth}
        \includegraphics[width=\linewidth]{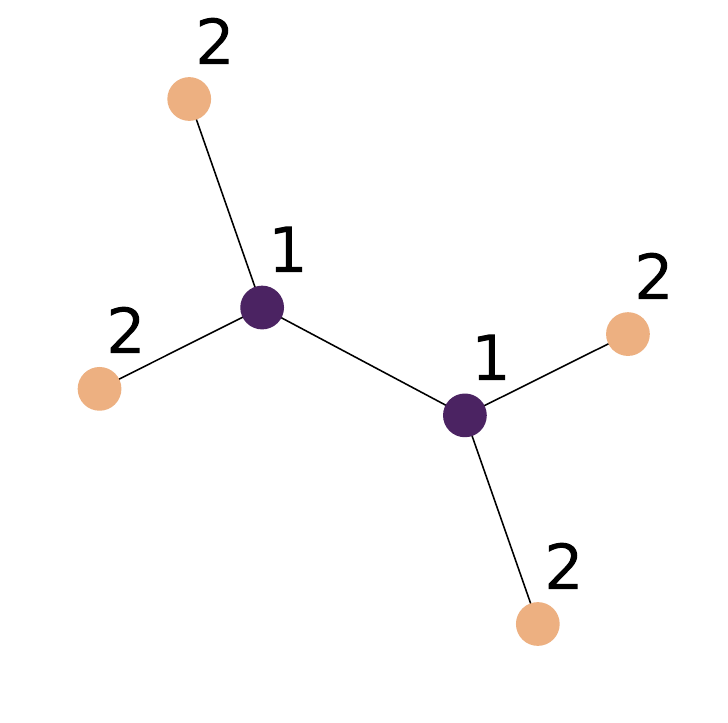}
    \end{subfigure}
    \begin{subfigure}{0.45\linewidth}
        \includegraphics[width=\linewidth]{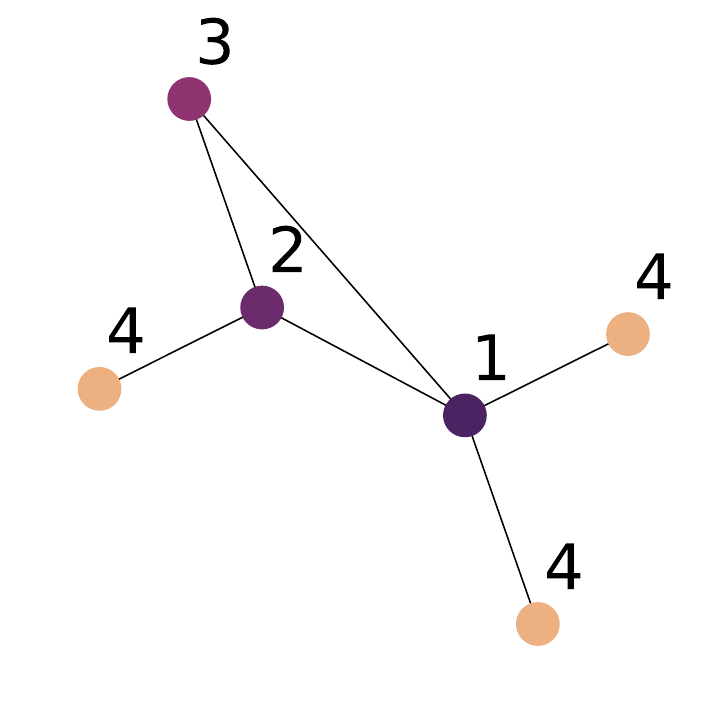}
    \end{subfigure}
    \caption{Example of a graph where adding an edge decreases the rank of magnitude homology centrality for $l=3$. The node labels are the ranks, adding the edge in the right graph decreases the rank of the top vertex.}
    \label{fig:counterex_rank_increase_axiom}
\end{wrapfigure}
\textcite{boldi2014axioms_centrality} discuss a version of the endpoint increase axiom that only requires the rank obtained from the centrality of the incident node not to decrease, in this variation we found that all four versions of the (Eulerian) magnitude homology centrality do not hold, see \Cref{fig:counterex_rank_increase_axiom} for a counterexample.

Nevertheless, there is also a weaker version where only the most central nodes are studied under the addition of incident edges. The so-called top node axiom, as stated in \textcite{meshcheryakova2024comparative_analysis_centrality}, forms a slight modification of \textcite{sabidussi1966centrality_index}.
\begin{defn}[Top node axiom]
    The top node axiom is satisfied if for any graph $G$ and any vertex $v\in G$ with the highest centrality measure, adding any edge $e$ incident to $v$ does not decrease its rank, i.e. $v$ still has the highest centrality in the modified graph.
\end{defn}

We check the top node axiom on the same $1253$ graphs from the graph atlas and find that for any of the versions of the magnitude homology centrality the axiom holds. To do some further testing, we check on the $441$ connected graphs with less than or equal to $30$ vertices in the `Netzschleuder' \cite{peixeto2020netzschleuder} catalogue (see \Cref{sec:comparison_existing_cent_measures} where we use these graphs as well) and find that also on those graphs the axiom holds true for all tested versions of (Eulerian) magnitude homology centrality.
We hence conjecture that the top node axiom is satisfied for the Eulerian and standard version of magnitude homology centrality in degrees $l\in \{2,3\}$.

\section{Results}
\label{sec:results}
For this section, we investigate further how the (Eulerian) magnitude homology centrality behaves. Concretely, in \Cref{sec:rook_shrikh} we demonstrate how the proposed centralities can be used to compare two different graphs. \Cref{sec:comparison_existing_cent_measures} gives a comparison between our proposed centrality measures based on magnitude homology and a selection of existing centrality measures. The comparison shows that the (Eulerian) magnitude centralities agree to some extent with existing measures but also offer a new way to defining central nodes in a graph.
\subsection{Distinguishing Rook and Shrikhande Graph}
\label{sec:rook_shrikh}
We want to investigate how our proposed centrality measures can be used to study the structure of two different graphs. In particular, we want to investigate its relation to distinguishing non-isomorphic graphs and demonstrate the capability on a specific example. The problem of distinguishing non-isomorphic graphs is discussed in machine learning, where graph models are often investigated with regards to their expressivity \cite{xu2019how_powerful_gnns}. An often used measure for the expressivity is the Weisfeiler--Leman (WL) isomorphism test. This test iteratively colours vertices in a graph based on information from their neighbouring nodes, which can give proof that two graphs are not isomorphic. The test does not distinguish all  pairs of non-isomorphic graphs, leading to pairs of graphs that are so-called 1-WL indistinguishable. There also exist higher-dimensional Weisfeiler-Leman tests, denoted $k$-WL tests, a sequence of isomorphism tests which strictly increase in expressivity for increasing $k$. For more details regarding the WL test we refer to \textcite{morris2023wl_go_ml}.

We have seen that the (Eulerian) magnitude centrality is a local measure
and therefore cannot assign different centrality values to vertices in a graph if the neighbourhoods $G_v^l$ and $G_v^l \setminus \{v\}$ are isomorphic for all vertices $v \in G$. Conversely, if two graphs have different magnitude centrality values for some vertices, say vertices $v \in G$ and $u \in H$ for some graphs $G$ and $H$ with $\MHcent{G, v, l} \neq \MHcent{H, u, l}$, then we can conclude that the induced graphs are not isomorphic, that is $G_v^l \ncong H_u^l$ or $G_v^l \setminus \{v\} \ncong H_u^l \setminus \{u\}$, meaning the neighbourhoods of $v$ and $u$ are structurally different.

We demonstrate this property on the two non-isomorphic strongly regular graphs with parameters $(16,6,2,2)$, namely the rook $4 \times 4$ graph $R(4,4)$ and the so-called Shrikhande graph $\mathcal{S}$. They are known to be difficult to distinguish by the WL-test. Concretely, they are indistinguishable by $1$-WL since they are regular graphs with the same degree, furthermore, they are indistinguishable by $3$-WL but distinguishable by $4$-WL \cite{arvind2020weisfeiler_leman_invariance, balcilar2021limits_message_passing_gnns}. Both graphs are vertex-transitive \cite{coolsaet2023house_of_graphs}, so the centrality measures cannot distinguish between the vertices of each graph. For any vertex $v \in R(4,4)$ and any vertex $u \in \mathcal{S}$ we have the following magnitude centrality values of the corresponding graphs:
\begin{center}
    \begin{tabular}{ccccc}
    \toprule
    &$\MHcent{-, v, 2}$ & $\MHcent{-, v, 3}$ & $\EMHcent{ - , u, 2}$ & $\EMHcent{-, u, 3}$\\
    \midrule
    $R(4,4)$ & $97$ & $553$ & $85$ & $385$ \\
    $S$ & $97$ & $529$ & $85$ & $265$\\
    \bottomrule
\end{tabular}
\end{center}
We can see that for $l=2$, the magnitude centrality values between the $4\times4$ rook graph and the Shrikhande graph are \emph{equal}, but for $l=3$ they are different, meaning that we can deduce that the structure between the neighbourhoods of the two graphs must be different. Note that the diameter of both rook and Shrikhande graph is $2$, so the $2$-neighbourhoods of any vertices contain already the entire graph.

\subsection{Comparison to Existing Centrality Measures}
\label{sec:comparison_existing_cent_measures}
In this section, we want to understand the behaviour of the (Eulerian) magnitude homology centrality on different graphs with respect to pre-existing centrality measures. Following \textcite{valente2008how_correlated}, we compute the correlation between the different centrality measures and ideally there is positive correlation between all centrality measures because they are designed to measure one concept. At the same time, some variation between the different measures is preferable since otherwise it could indicate that the newly proposed centrality measures are redundant.

\subsubsection{Centrality Measure Comparison Partners}
\label{sec:baseline_centrality_measures}
We choose the established centrality measures degree, betweenness \cite{freeman1977centrality_betweenness}, closeness \cite{bavelas1950communications_patterns_task_oriented_groups}, subgraph \cite{estrada2005subgraph_centrality}, and PageRank \cite{brin1998anatomy_large_sclae_hypertextual_web_search} as baselines to compare against. These centrality measures stem from a variety of different approaches: \textcite{valente2008how_correlated} categorize the closeness and degree centrality as geometric measures since they are using information of how many nodes exist at a given distance; PageRank as a spectral measure because it uses the eigenvector of some matrix obtained from the graph, this measure is one of the most widely used centrality measures as it is used in the ranking of web pages; Betweenness is path based by measuring on how many shortest path a given node lies. We additionally compare to the subgraph centrality, another spectral centrality measure \cite{menara2024computing_eulerian_mh} which is considering the number of closed walks starting and ending at each node with the idea of capturing the number of subgraphs that the node takes part in \cite{estrada2005subgraph_centrality}. Formal definitions of the  chosen baseline centralities can be found for example in \textcite{saxena2020centrality_measures_complex_networks, estrada2005subgraph_centrality}. We use default parameters for all these comparison partners, specifically, we set the damping parameter $\alpha = 0.85$ for PageRank centrality, which was observed to perform well in real-world networks and widely used \cite{saxena2020centrality_measures_complex_networks}.

\subsubsection{Zachary Karate Club and Frucht Graph}
As an initial check, we analyse individual graphs and compute the Pearson correlation coefficients between the standard and Eulerian magnitude homology centrality for $l \in \{2,3\}$ and the above described established centrality measures for the nodes in each graph. We find that for certain graphs, the correlation between all considered centrality measures is high. For example, in the Zachary karate club network \cite{zachary1977information_flow_model}, for which we show the centralities and correlation in \Cref{fig:cent_comparison_Zachary_Karate_Club}, the correlation between all considered centrality measures is higher than $0.84$, indicating that the different centrality measures mostly agree on what the most important nodes are.\begin{figure}[t]
    \begin{minipage}{0.49\textwidth}
    \begin{subfigure}{0.32\linewidth}
        \includegraphics[width=\linewidth]{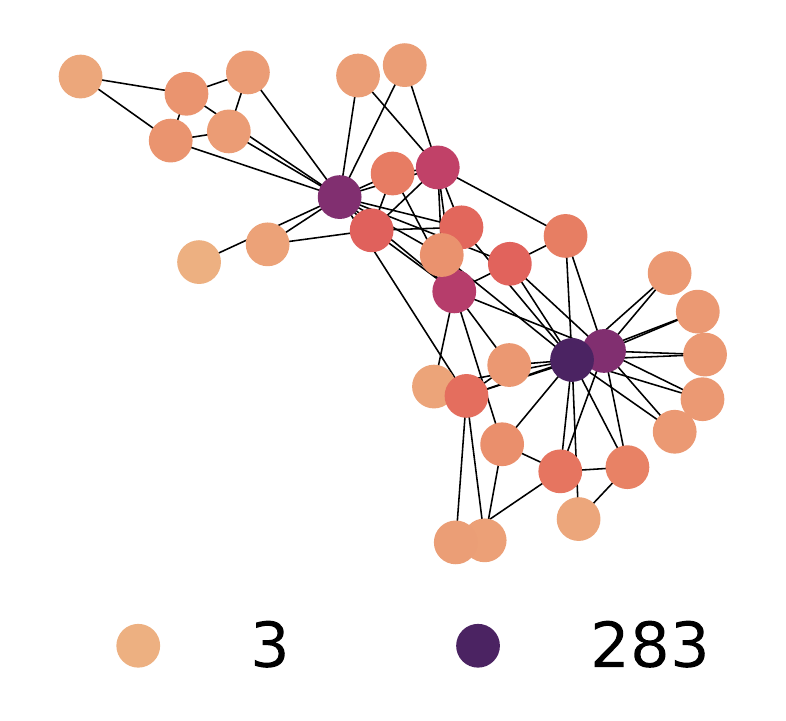}
        \caption{$C_{\mathrm{EMH}}$ with $l=2$.}
    \end{subfigure}
    \begin{subfigure}{0.32\linewidth}
        \includegraphics[width=\linewidth]{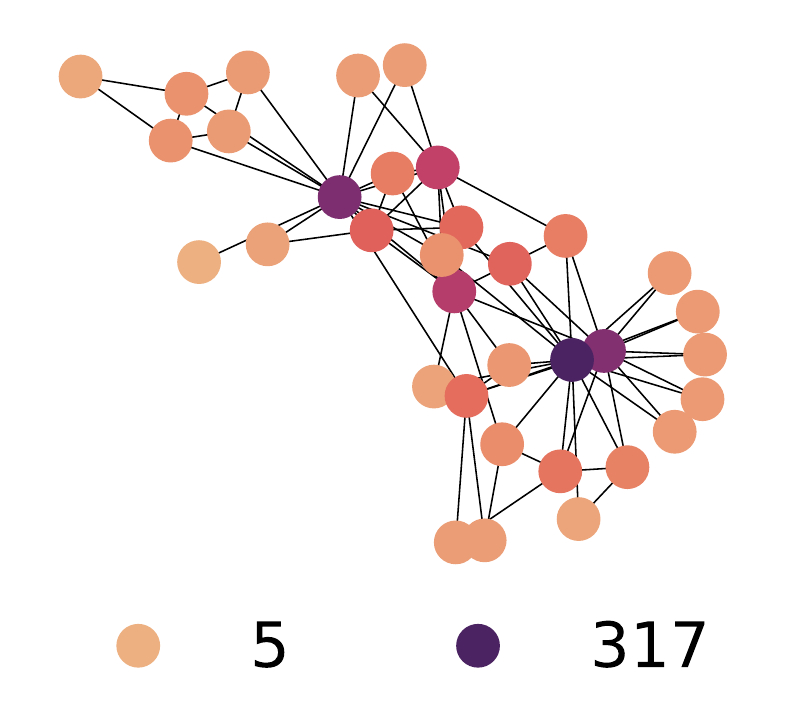}
        \caption{$C_{\mathrm{MH}}$ with $l=2$.}
    \end{subfigure}
    \begin{subfigure}{0.32\linewidth}
        \includegraphics[width=\linewidth]{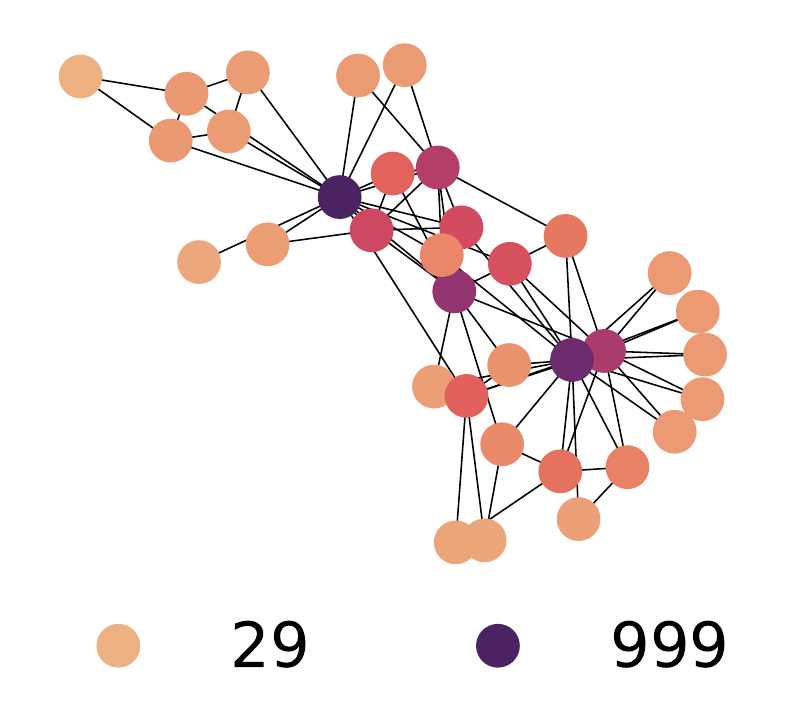}
        \caption{$C_{\mathrm{EMH}}$ with $l=3$.}        
    \end{subfigure}\\
    \begin{subfigure}{0.32\linewidth}
        \includegraphics[width=\linewidth]{Figures/centrality_plots_Zachary_Karate_Club/EMH_centrality_l=3.pdf}
        \caption{$C_{\mathrm{MH}}$ with $l=3$.}
    \end{subfigure}
    \begin{subfigure}{0.32\linewidth}
        \includegraphics[width=\linewidth]{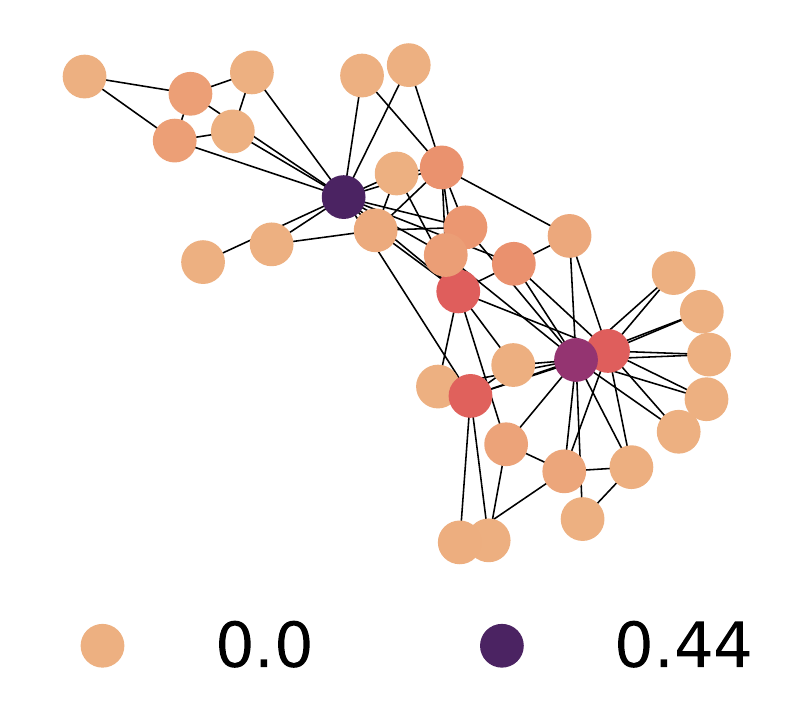}
        \caption{Betweenness centrality.}
    \end{subfigure}
    \begin{subfigure}{0.32\linewidth}
        \includegraphics[width=\linewidth]{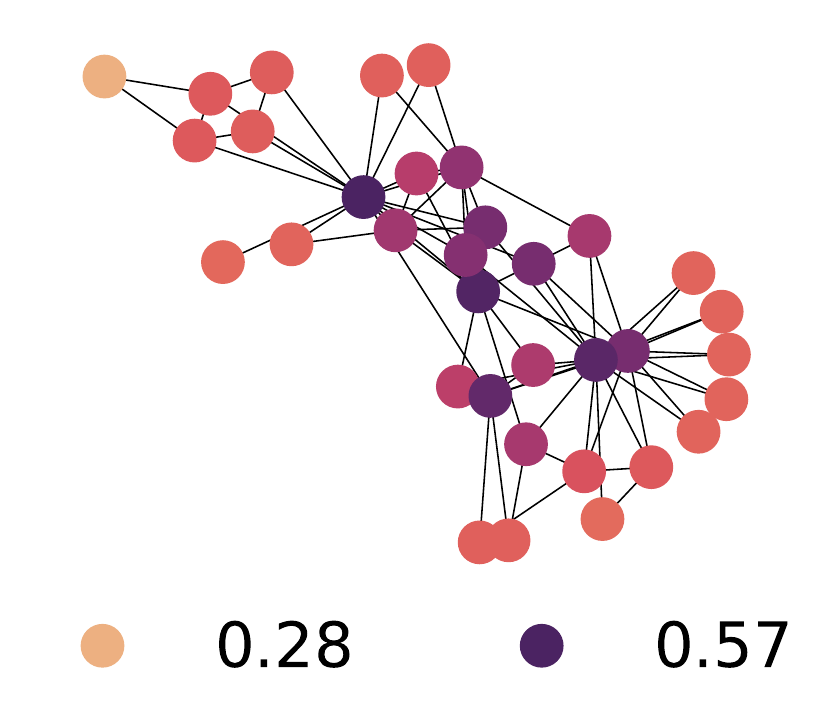}
        \caption{Closeness centrality.}
    \end{subfigure} \\
    \begin{subfigure}{0.32\linewidth}
        \includegraphics[width=\linewidth]{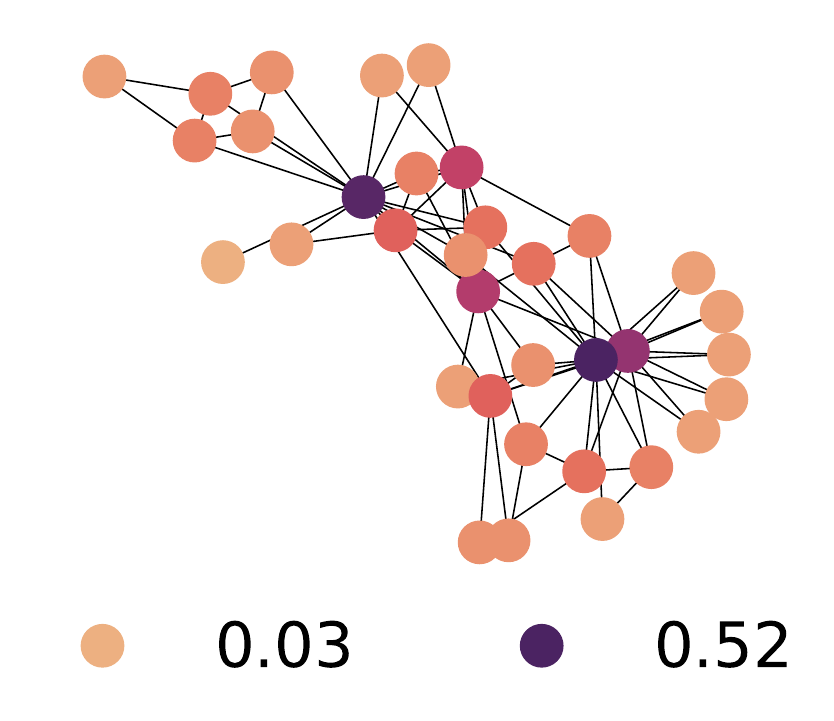}
        \caption{Degree centrality.}
    \end{subfigure}
    \begin{subfigure}{0.32\linewidth}
        \includegraphics[width=\linewidth]{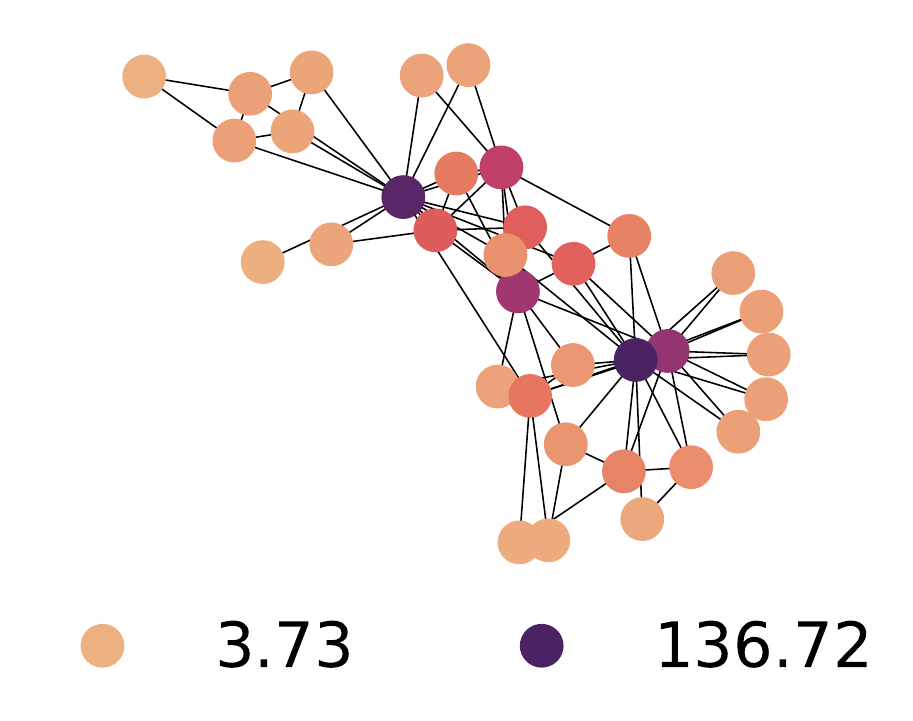}
        \caption{Subgraph centrality.}
    \end{subfigure}
    \begin{subfigure}{0.32\linewidth}
        \includegraphics[width=\linewidth]{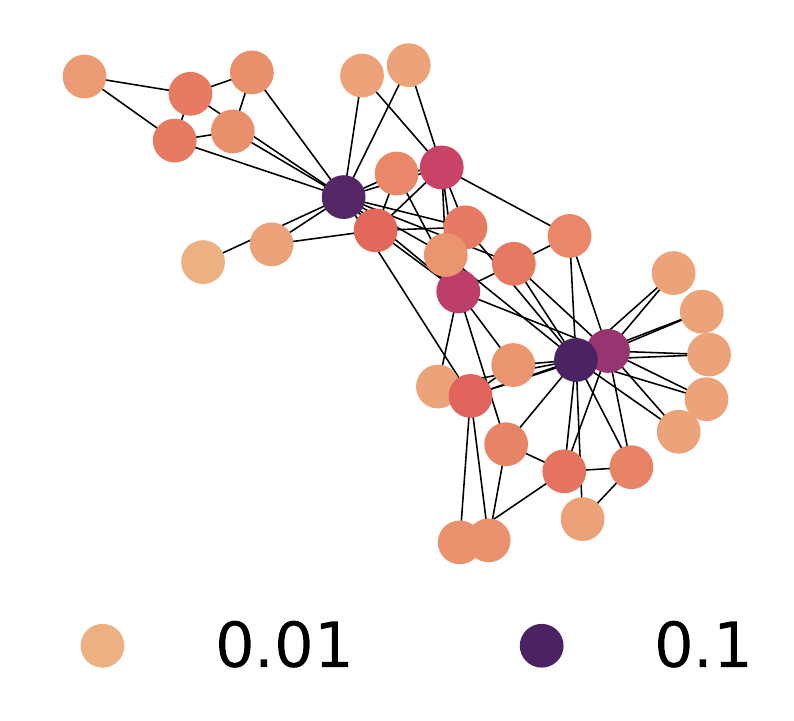}
        \caption{PageRank centrality.}
    \end{subfigure} 
    \end{minipage} \hfill
        \begin{minipage}[c]{0.49\textwidth}
        \begin{subfigure}{\linewidth}
            \centering
            \includegraphics[width=\linewidth]{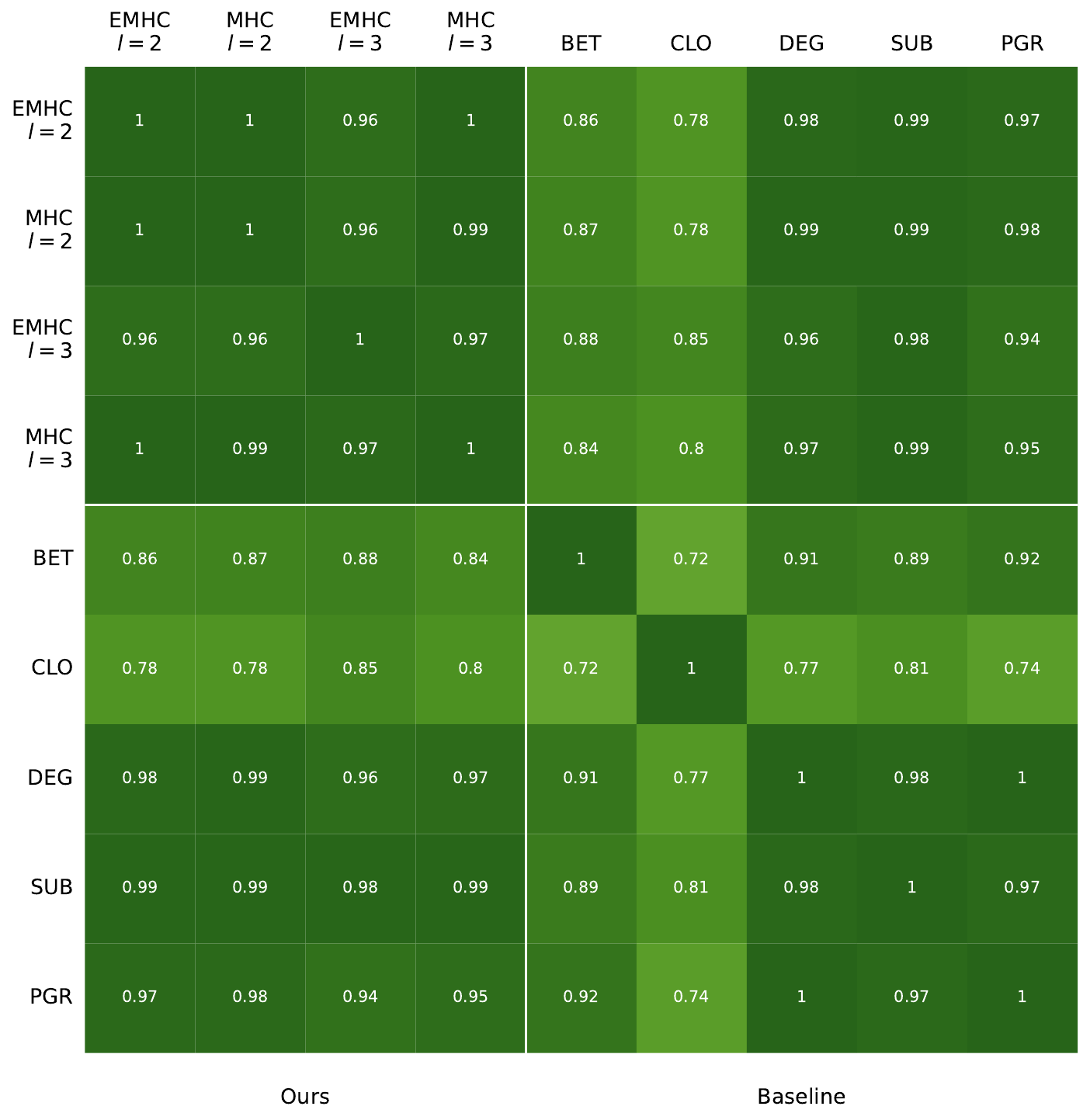}
            \caption{Pearson correlation coefficients for measuring the correlation between different centrality measures. A correlation of $+1$~(or $-1$, respectively) indicates that two measures are strongly (anti-)correlated meaning that there is a \emph{linear} relationship between the two of them.}
        \end{subfigure}
    \end{minipage}
    \caption{Comparison of different centrality measures for Zachary's Karate Club network.}
    \label{fig:cent_comparison_Zachary_Karate_Club}
\end{figure}  In other graphs, such as the Frucht graph seen in \Cref{fig:cent_comparison_Frucht}, we observe different correlations, ranging from $-0.87$ between Eulerian magnitude homology centrality with $l=2$ and Betweenness centrality over $0.35$ between Eulerian magnitude homology with $l=3$ and Betweenness centrality to $0.87$ between Eulerian magnitude homology centrality with $l=2$ and Subgraph Centrality. These examples demonstrate that we capture different aspects of the importance of nodes with the (Eulerian) magnitude homology centrality. 
\begin{figure}[t]
    \begin{minipage}{0.49\textwidth}
        \begin{subfigure}{0.32\linewidth}
            \includegraphics[width=\linewidth]{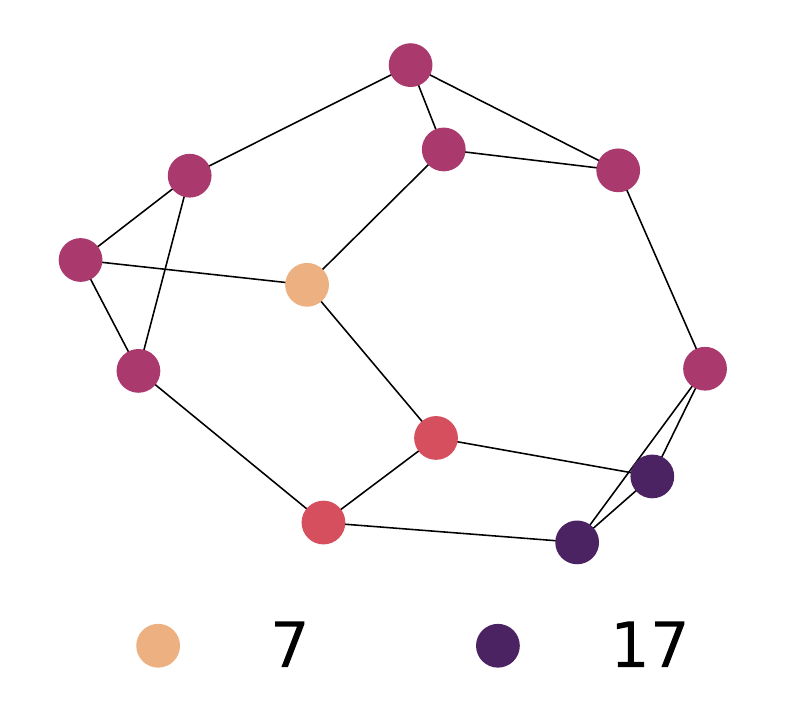}
            \caption{$C_{\mathrm{EMH}}$ with $l=2$.}
        \end{subfigure}
        \begin{subfigure}{0.32\linewidth}
            \includegraphics[width=\linewidth]{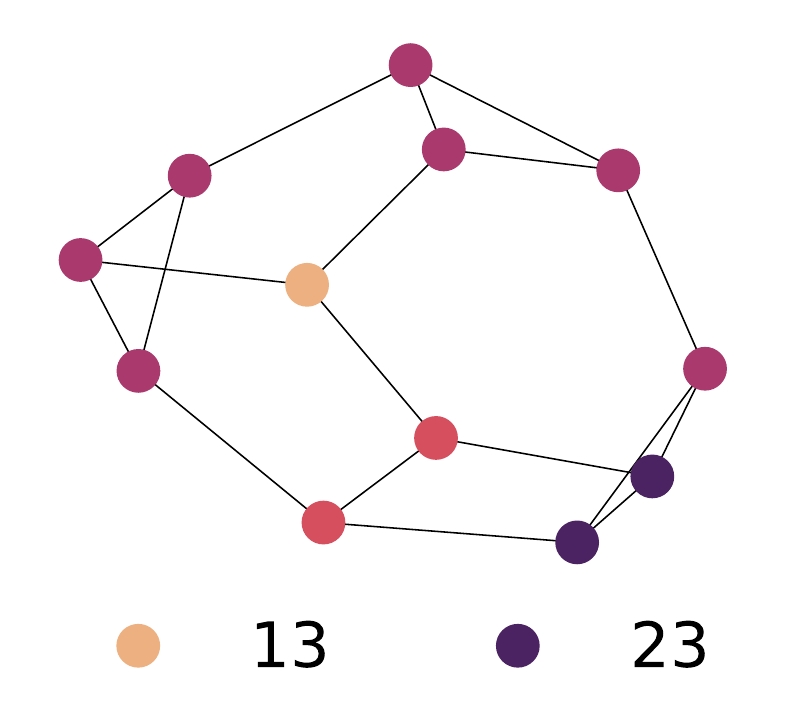}
            \caption{$C_{\mathrm{MH}}$ with $l=2$.}
        \end{subfigure}
        \begin{subfigure}{0.32\linewidth}
            \includegraphics[width=\linewidth]{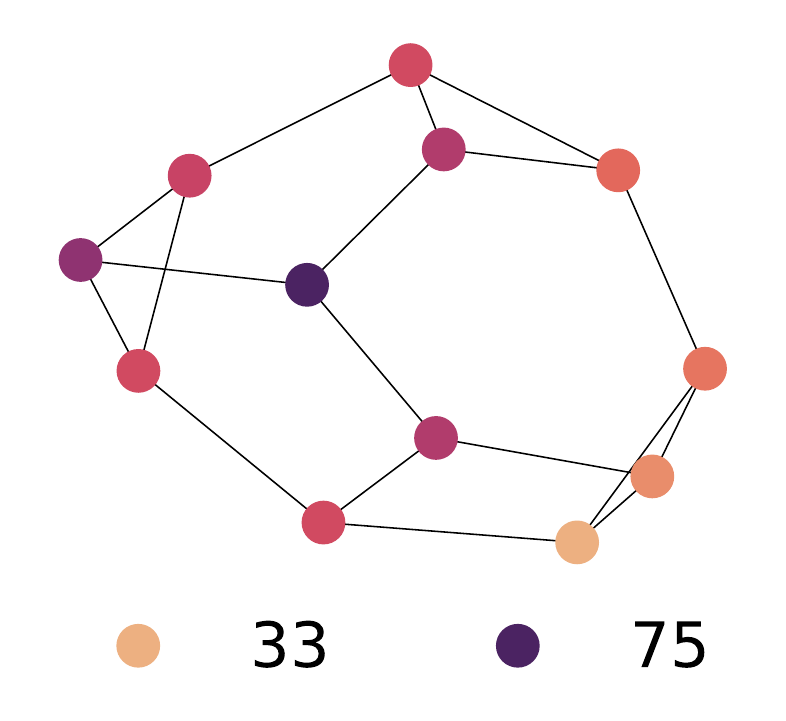}
            \caption{$C_{\mathrm{EMH}}$ with $l=3$.}        
        \end{subfigure}\\
        \begin{subfigure}{0.32\linewidth}
            \includegraphics[width=\linewidth]{Figures/centrality_plots_frucht/EMH_centrality_l=3.pdf}
            \caption{$C_{\mathrm{MH}}$ with $l=3$.}
        \end{subfigure}
        \begin{subfigure}{0.32\linewidth}
            \includegraphics[width=\linewidth]{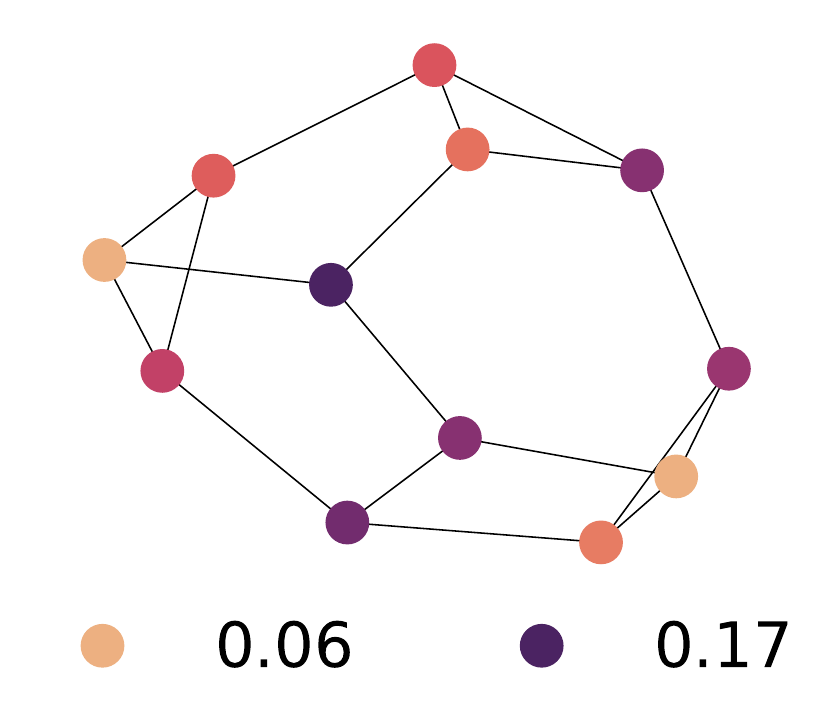}
            \caption{Betweenness centrality.}
        \end{subfigure}
        \begin{subfigure}{0.32\linewidth}
            \includegraphics[width=\linewidth]{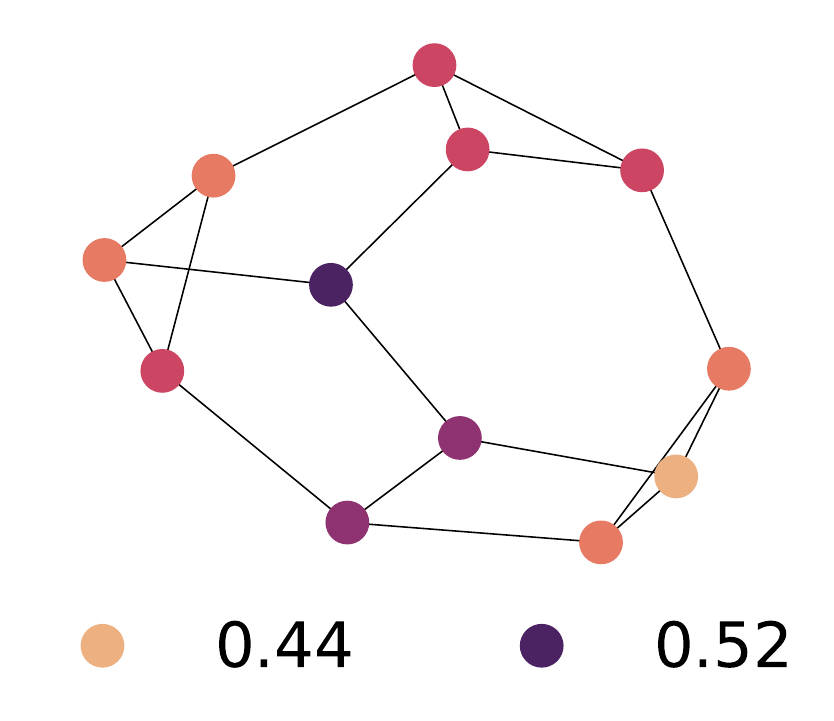}
            \caption{Closeness centrality.}
        \end{subfigure} \\
        \begin{subfigure}{0.32\linewidth}
            \includegraphics[width=\linewidth]{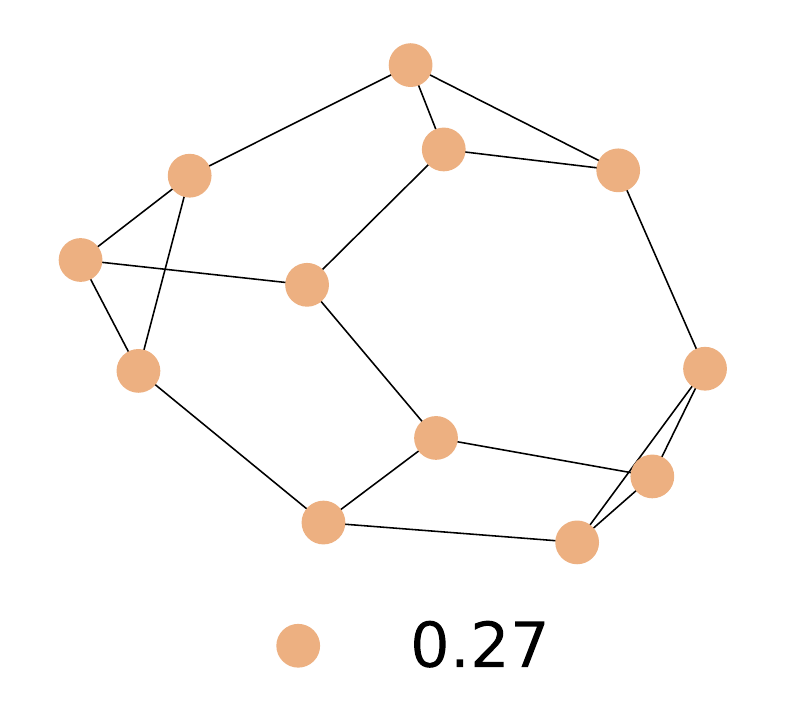}
            \caption{Degree centrality.}
        \end{subfigure}
        \begin{subfigure}{0.32\linewidth}
            \includegraphics[width=\linewidth]{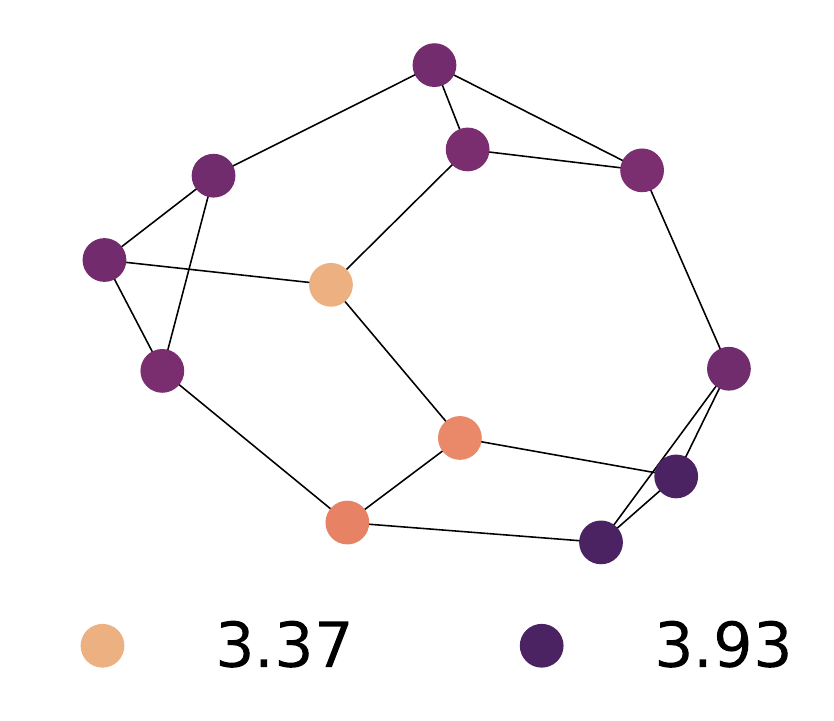}
            \caption{Subgraph centrality.}
        \end{subfigure}
        \begin{subfigure}{0.32\linewidth}
            \includegraphics[width=\linewidth]{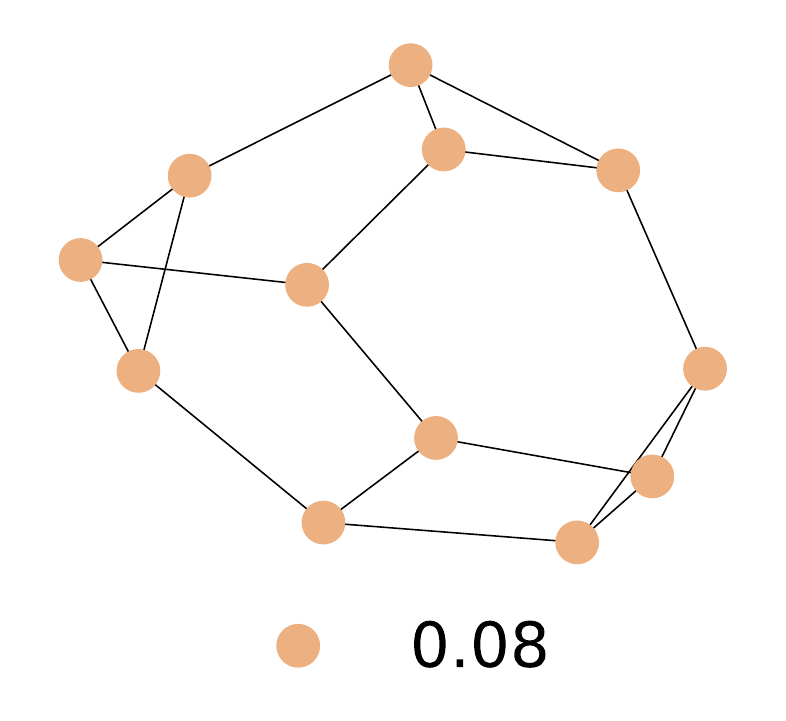}
            \caption{PageRank centrality.}
        \end{subfigure} 
    \end{minipage} \hfill
    \begin{minipage}[c]{0.49\textwidth}
        \begin{subfigure}{\linewidth}
            \centering
            \includegraphics[width=\linewidth]{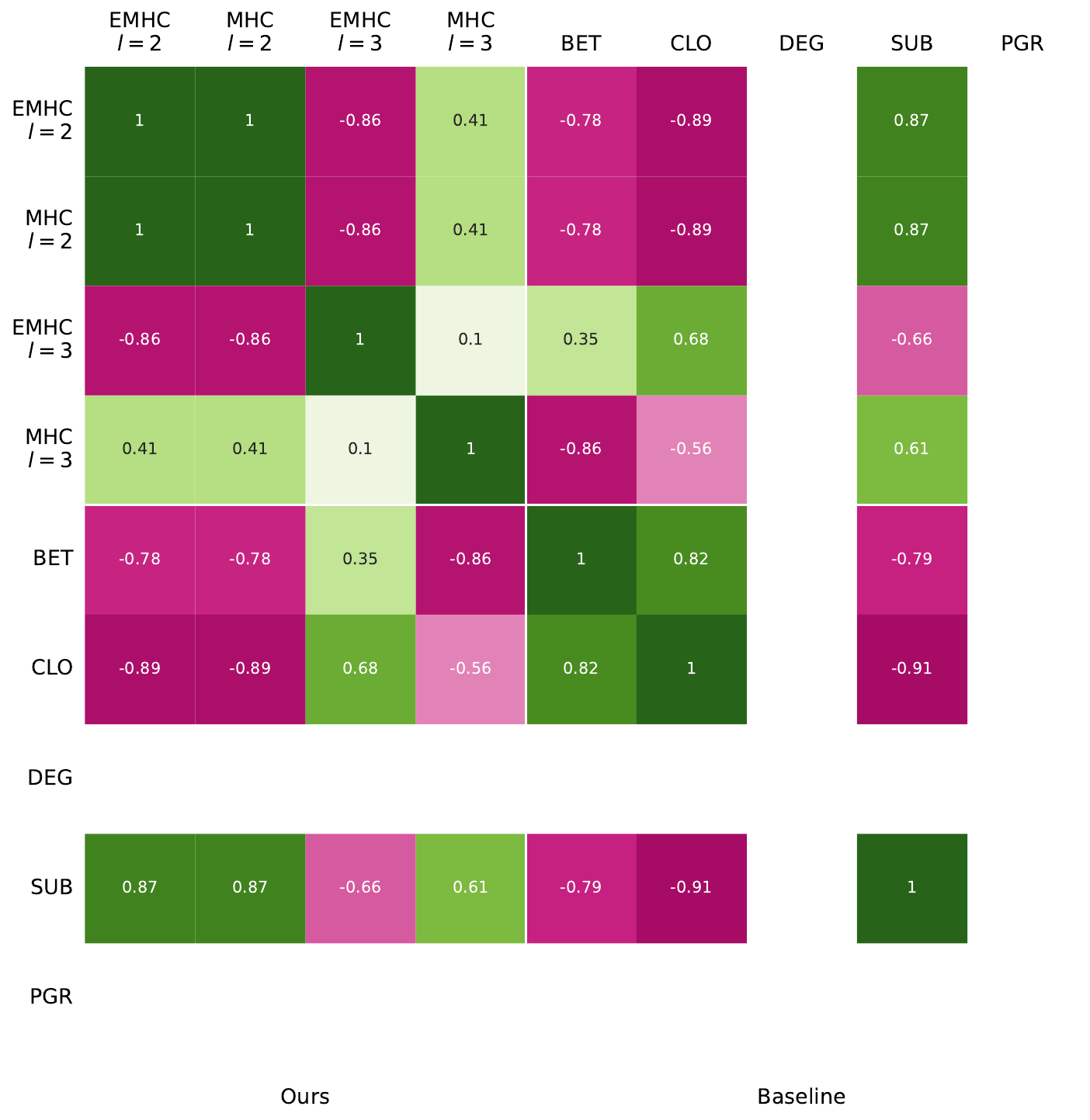}
            \caption{Pearson correlation coefficients; if a centrality measure is constant over all vertices in the graph, we do not report any correlation.}       
        \end{subfigure}
    \end{minipage}
    \caption{Comparison of different centrality measures for the Frucht graph.}
    \label{fig:cent_comparison_Frucht}
\end{figure} 
\subsubsection{Correlation over Various Network Datasets}
To understand the comparison to existing centrality measures in a more principled manner, we investigate the correlation of the (Eulerian) magnitude homology centrality to the baseline centrality measures over a variety of different graphs. We consider all connected undirected graphs from the `Netzschleuder' network catalogue \cite{peixeto2020netzschleuder} that have less than or equal to $30$ nodes, leaving us with $441$ graphs \cite{atran2009philippinesambassador,mcauley2012discovering_social_circles,knight2011internet_topology_zoo,grant_1973_dominance,kaminski2018moviegalaxies,read1954cultures_central_highlands,rhodes2009inferring_missing_links,hobson2013evolution_modern_capitalism,gelardi2020measuring_social_networks_primates}. The restriction to connected graphs is due to the fact that the closeness centrality in its original form is only defined for connected graphs \cite{saxena2020centrality_measures_complex_networks}. For each of these graphs we compute the Pearson correlation coefficient between the centrality measures of all nodes and average them over all graphs; the resulting correlations with standard deviation are reported in \Cref{fig:avg_correlation}.
When comparing the correlation of our proposed (Eulerian) magnitude homology centrality, we note the following: The correlation between all versions of the (Eulerian) magnitude homology centrality is on average at least 0.96, concretely for $l=2$ there does not seem to be an empirical difference between using the Eulerian versus the standard version. In the comparison to the existing centrality measures, all the computed variations of (Eulerian) magnitude homology centrality correlate highest with the subgraph centrality, both of these in some way counting specific subgraphs structure that a node participates in. The lowest correlation of all variants of (Eulerian) magnitude homology centrality with existing centrality measures is with the betweenness centrality, ranging from $0.73$ to $0.78$.
We find an overall high average correlation over different centrality measures for the studied set of graphs. We want to point out that \textcite{valente2008how_correlated} have also computed the average correlation of centrality measures of $58$ networks obtained from different studies on bounded communities. The correlations they report are lower than the ones we obtain when comparing the correlations included in both studies. This leads us to conclude that  on the set of graphs we have chosen,  established centrality measures tend to agree much more strongly concerning which nodes are the most central.
\begin{figure}
\begin{subfigure}[t]{0.48\textwidth}
    \centering
    \includegraphics[width=\linewidth]{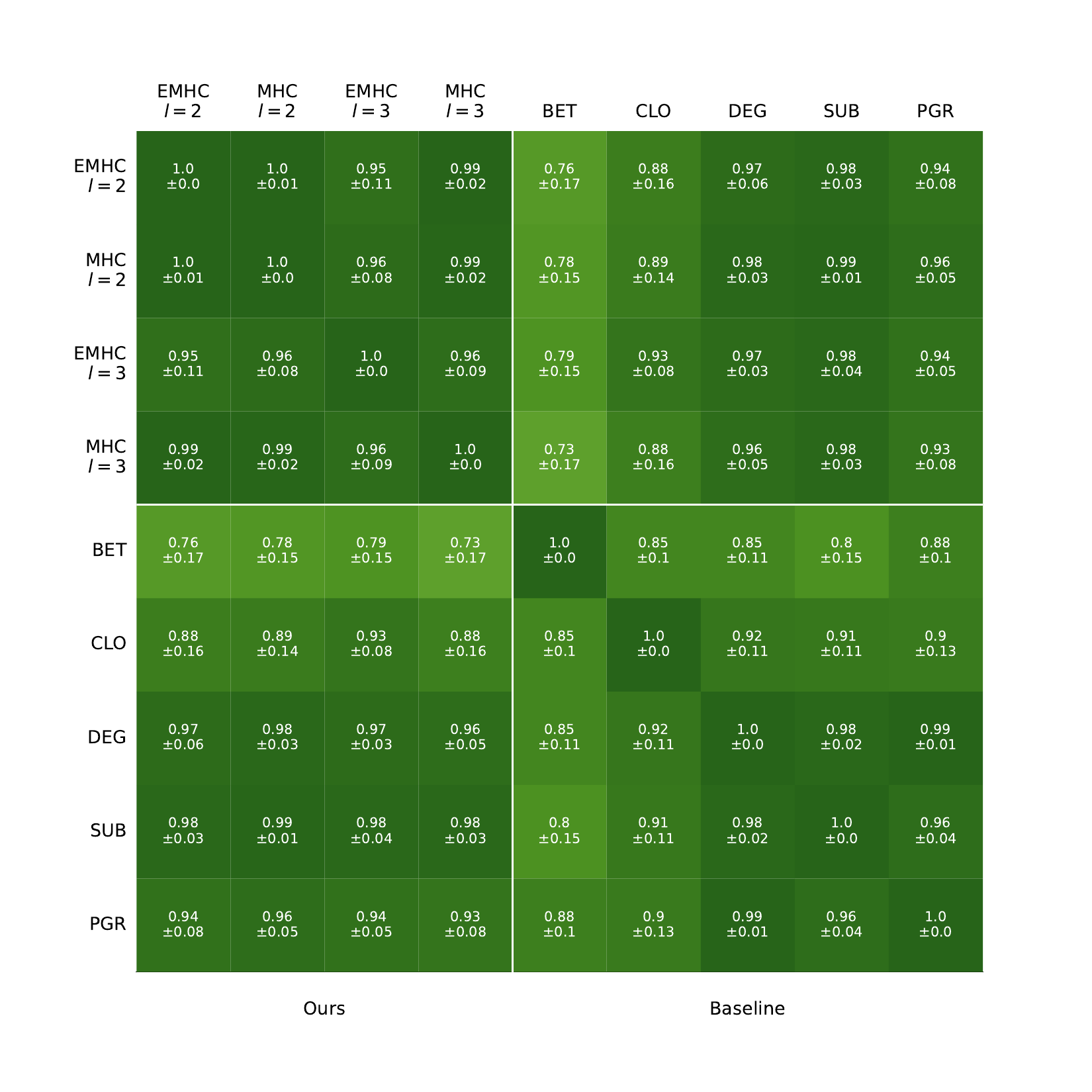}
    \caption{Average Pearson correlation coefficient over each connected graph in the `Netzschleuder' catalogue restricted to less than or equal to $30$ vertices.}
    \label{fig:avg_correlation}
\end{subfigure}
\hfill \begin{subfigure}[t]{0.48\textwidth}
    \centering
    \includegraphics[width=\linewidth]{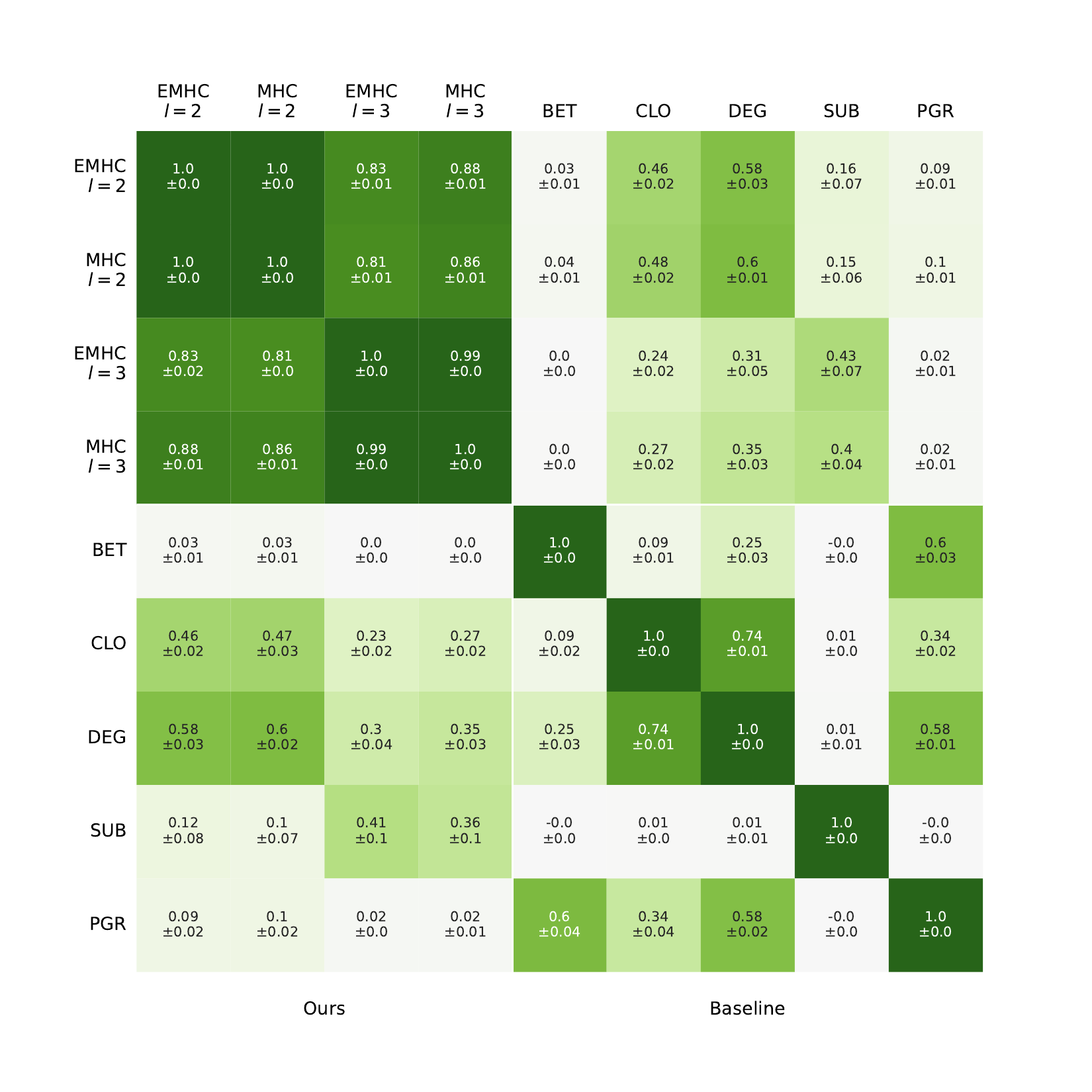}
    \caption{Mean and standard deviation of the $R^2$-score from training a linear regression with $5$-fold cross-validation on the centrality measures of all vertices of the connected graphs in the `Netzschleuder' catalogue restricted to less than or equal to $30$ vertices.}
    \label{fig:lin_reg}
\end{subfigure}
\label{fig:comparison}
\caption{Comparison between our proposed centrality measures (Eulerian) magnitude homology centrality (EMHC and MHC respectively) for $l=2$ and $l=3$ and existing centrality measures betweenness(BET), closeness (CLO), degreee (DEG), subgraph (SUB), and PageRank (PGR)}
\end{figure}

To further compare the different centrality measures, we generalize the correlation to predicting the centrality value of one measure with another by using a linear regression training on the set of vertices we obtain from the different graphs. Again, we compute the centrality values for al $441$ connected graphs from the `Netzschleuder' catalogue with less than or equal to $30$ vertices and concatenate the values of all vertices in the graphs. We then perform a $5$-fold cross-validation with an initial shuffle of the data and fit a linear regression model. We report the average and standard deviation of the coefficient of determination $R^2$ over the different splits in \Cref{fig:lin_reg}.
The results show that empirically the Eulerian and standard magnitude homology centralities can be predicted well from each other. Furthermore, the change from $l=2$ to $l=3$ does still give an average $R^2$-score of at least $0.81$, showing that there is some new information when considering the higher dimensions in the tested empirical datasets, but we do already see a strong correlation between the two parameter choices. From the existing centrality measures, we observe that the degree and closeness centrality measure can be predicted with an $R^2$-score of $0.74$ from each other, forming another cluster in the heatmap. Between our proposed magnitude homology centrality measures and the established measures, the highest $R^2$ score occurs with the degree centrality and standard and Eulerian magnitude homology with $l=2$. This might reflect the relation we have shown in \Cref{prop:EMHC_l=0}. Comparing the rest of the centrality measures, the results do not indicate a clear linear trend between the proposed (Eulerian) magnitude homology centralities and the baselines, indicating that (Eulerian) magnitude homology centrality indeed captures different, novel aspects of vertex centrality.
Overall, we conclude that our proposed centrality measures are aligned with established centrality measures on certain graphs, while still providing new insights into the notion of centrality.

\section{Discussion}

We introduced a novel family of vertex centrality measure by using the magnitude homology and its Eulerian version. We showed that these measure are local and thus obtained the possibility to tune the range of the neighbourhood that is taken into account when computing the centrality of a node, allowing for flexibility and control of the computational complexity. Furthermore, we provided evidence (partially formally and partially empirically) that the proposed measures satisfy properties that are characteristic for a centrality measure, and we performed extensive comparison against established centrality measures. For obtaining the results presented in this paper, we have implemented a na\"ive algorithm to compute the (Eulerian) magnitude homology centralities. This algorithm is still limited in terms of graph size but we believe that there is great potential for future improvements, for instance via approximation schemes.

Our work also provides several other follow-up directions that will help deepen our understanding of the proposed centrality measures. Concretely, we have observed that in many regular graphs, the (Eulerian) magnitude homology centrality can offer a more refined ranking of the nodes than just the degree. Further investigations into the discriminative power of the proposed centrality measures remain to be explored. Starting from the observed correlation between our proposed centrality measures and the subgraph centrality, we assume relation between our proposed centrality measures and the spectrum of the graph can be further studied, potentially gaining understanding about the relation to subgraph counting and the subgraph centrality in particular.

\section*{Acknowledgements}

The authors thank Katharina Limbeck and Kavir Sumaraj for their helpful feedback. This work has received funding from the Swiss State Secretariat for
Education, Research, and Innovation~(SERI).

\printbibliography

@misc{magnitude_evaluation_diversity_latent_space,
      title={Metric Space Magnitude for Evaluating the Diversity of Latent Representations}, 
      author={Limbeck, Katharina and Andreeva, Rayna and others},
      howpublished = {28th Conference on Neural Information Processing Systems (NeurIPS 2024)},
      year={2024},
}

@article{kaneta2021mh_metric_spaces_order_complexes,
      title={Magnitude Homology of Metric Spaces and Order Complexes}, 
      author={Kaneta, Ryuki and Yoshinaga, Masahiko},
      year={2021},
      journal = {Bulletin of the London Mathematical Society},
      volume = {53},
      number = {3},
      pages = {893-905},
}

@misc{magnitude_and_generalisations_in_NN,
      title={Metric Space Magnitude and Generalisation in Neural Networks}, 
      author={Andreeva, Rayna and Limbeck, Katharina and others},
      year={2023},
      eprint={2305.05611},
      archivePrefix={arXiv},
}

@misc{menara2024computing_eulerian_mh,
      title={Computing Eulerian Magnitude Homology}, 
      author={Menara, Giuliamaria and Manzoni, Luca},
      year={2024},
      eprint={2410.10376},
      archivePrefix={arXiv},
}

@misc{caputi2025emh_diagonality_injective_words_regular_path_homology,
      title={Eulerian Magnitude Homology: Diagonality, Injective Words, and Regular Path Homology}, 
      author={Caputi, Luigi and Menara, Giuliamaria},
      year={2025},
      eprint={2503.06722},
      archivePrefix={arXiv},
}

@misc{gu2018magnitude_homology_morse,
      title={Graph Magnitude Homology via Algebraic Morse Theory}, 
      author={Gu, Yuzhou},
      year={2018},
      eprint={1809.07240},
      archivePrefix={arXiv},
}

@misc{giusti2024eulerianMH_subgraph_structure,
      title={Eulerian Magnitude Homology: Subgraph Structure and Random Graphs}, 
      author={Gusti, Chad and Menara Giuliamaria},
      year={2024},
      eprint={2403.09248},
      archivePrefix={arXiv},
}

@article{hepworth2017categorifying,
	year = {2017},
	publisher = {International Press of Boston},
	volume = {19},
	number = {2},
	pages = {31--60},
	author = {Richard Hepworth and Simon Willerton},
	title = {Categorifying the magnitude of a graph},
	journal = {Homology, Homotopy and Applications}
}

@article{asao2024girth_mh_diagonality,
    author = {Asao, Yasuhiko and Hiraoka, Yasuaki and Kanazawa, Shu},
    title = {Girth, magnitude homology and phase transition of diagonality},
    journal = {Proceedings of the Royal Society of Edinburgh},
    volume = {154},
    pages = {221--247},
    year = {2024},
}

@article{morris2023wl_go_ml,
    author = {Christopher Morris and Yaron Lipman and Haggai Maron and Bastian Rieck and Nils M. Kriege and Martin Grohe and Matthias Fey and Karsten Borgwardt},
    title = {Weisfeiler and Leman go Machine Learning: The Story so far},
    journal = {Journal of Machine Learning Research},
    year = {2023},
    volume  = {24},
    number  = {333},
    pages   = {1--59},
}

@article{valente2008how_correlated,
    author = {Valente, Thomas W. and Coronges, Kathryn and Lakon, Cynthia and Costenbader, Elizabeth},
    title = {How Correlated Are Network Centrality Measures?},
    journal = {Connect (Tor)},
    year = {2008},
    pages = {16--26},
    volume = {28},
    number = {1},
}

@misc{saxena2020centrality_measures_complex_networks,
      title={Centrality Measures in Complex Networks: A Survey}, 
      author={Saxena, Akrati and Iyengar, Sudarshan},
      year={2020},
      eprint={2011.07190},
      archivePrefix={arXiv},
}

@article{leinster2013magnitude_metric,
    author = {Leinster, Tom},
    title = {The magnitude of metric spaces},
    journal = {Documenta Mathematica},
    year = {2013},
    volume = {18},
    pages = {857--905}
}

@book{leinster2021entropy_diversity,
    author = {Leinster, Tom},
    title = {Entropy and Diversity},
    publisher = {Cambridge University Press},
    year = {2021},
}

@article{leinster2019magnitude_graph,
    author = {Leinster, Tom},
    title = {The magnitude of a graph},
    journal = {Mathematical Proceedings of the Cambridge Philosophical Society},
    year = {2019},
    volume={166},
    number={2},
    pages={247–264},
}

@article{coolsaet2023house_of_graphs,
	title = {House of {G}raphs 2.0: {A} database of interesting graphs and more},
	volume = {325},
	journal = {Discrete Applied Mathematics},
	author = {Coolsaet, Kris and D’hondt, Sven and Goedgebeur, Jan},
	year = {2023},
	pages = {97--107},
    note = {Available at \url{https://houseofgraphs.org}}
}

@misc{peixeto2020netzschleuder,
    author = {Peixoto, Tiago P.},
    title = {The Netzschleuder network catalogue and repository},
    year = {2020},
    url = {https://networks.skewed.de/}
}

@book{weibel2013homological_algebra,
    author = {Weibel, Charles A.},
    title = {An introduction to homological algebra},
    publisher = {Cambridge University Press},
    year = {2013},
}

@article{boldi2014axioms_centrality,
       author = {Boldi, Paolo and Vigna, Sebastiano},
    title = {Axioms for centrality} ,
    journal = {Internet Mathematics},
    year = {2014},
    volume = {10},
    pages = {222--262},
}

@article{meshcheryakova2024comparative_analysis_centrality,
    author = {Meshcheryakova, N. and Shvydun, S} ,
    title = {A Comparative Analysis of Centrality Measures in Complex Networks},
    journal = {Automation and Remote Control},
    year = {2024} ,
    volume = {85},
    number = {8},
    pages = {658--695},
}

@article{leinster2021magnitude_homology_enriched_cat_metric_space,
    author = {Leinster, Tom and Shulman, Michael},
    title = {Magnitude homology of enriched categories and metric spaces},
    journal = {Algebraic \& Geometric Topology},
    volume = {21},
    pages = {2175--2221},
    year = {2021},
}

@misc{gomi2025magnitude_homology_geodesic,
      title={Magnitude homology of geodesic space}, 
      author={Gomi, Kiyonori},
      year={2025},
      eprint={1902.07044},
      archivePrefix={arXiv},
}

@article{freeman1977centrality_betweenness,
      title={A Set of Measures of Centrality Based on Betweenness}, 
      author={Freeman, Linton C.},
      year={1977},
      journal = {Sociometry},
      number = {1},
      volume = {40},
      pages = {35--41},
}

@article{sabidussi1966centrality_index,
    author = {Gert Sabidussi},
    title = {The centrality index of a graph},
    journal = {Psychometrika},
    year = {1966},
    volume = {31},
    number = {4},
    pages = {581--603},
}

@article{arvind2020weisfeiler_leman_invariance,
	title = {On {Weisfeiler}-{Leman} invariance: {Subgraph} counts and related graph properties},
	volume = {113},
	journal = {Journal of Computer and System Sciences},
    volume = {113},
	author = {Arvind, Vikraman and Fuhlbrück, Frank and Köbler, Johannes and Verbitsky, Oleg},
	year = {2020},
	pages = {42--59},
}

@InProceedings{balcilar2021limits_message_passing_gnns,
  title = 	 {Breaking the Limits of Message Passing Graph Neural Networks},
  author =       {Balcilar, Muhammet and Heroux, Pierre and Gauzere, Benoit and Vasseur, Pascal and Adam, Sebastien and Honeine, Paul},
  booktitle = 	 {Proceedings of the 38th International Conference on Machine Learning},
  pages = 	 {599--608},
  year = 	 {2021},
  editor = 	 {Meila, Marina and Zhang, Tong},
  volume = 	 {139},
}

@article{estrada2005subgraph_centrality,
  author = {Estrada, Ernesto and Rodr\'{\i}guez-Vel\'azquez, Juan A.},
  journal = {Physical Review E},
  volume = {71},
  pages = {056103},
  year = {2005},
}

@article{bavelas1950communications_patterns_task_oriented_groups,
    author = {Alex Bavelas},
    title = {Communication Patterns in Task‐Oriented Groups},
    journal = {The Journal of the Acoustical Society of America},
    year = {1950},
    volume = {22},
    number = {6},
    pages = {725--730}
}

@article{brin1998anatomy_large_sclae_hypertextual_web_search,
	title = {The anatomy of a large-scale hypertextual Web search engine},
	volume = {30},
	journal = {Computer Networks and ISDN Systems},
	author = {Brin, Sergey and Page, Lawrence},
	year = {1998},
	pages = {107--117},
}

@inproceedings{
xu2019how_powerful_gnns,
title={How Powerful are Graph Neural Networks?},
author={Keyulu Xu and Weihua Hu and Jure Leskovec and Stefanie Jegelka},
booktitle={International Conference on Learning Representations},
year={2019},
}

@article{zachary1977information_flow_model,
    author = {Zachary, Wayne W.},
    title = {An Information Flow Model for Conflict and Fission in Small Groups},
    journal = {Journal of Anthropological Research},
    year = {1977},
    volume = {33},
    number = {4},
}

@book{wegener2005complexity_theory,
    author = {Ingo Wegene},
    title = {Complexity Theory},
    publisher = {Springer Berlin, Heidelberg},
    year = {2005},
}

@article{hagberg2008networkx,
    author =       {Hagberg, Aric A. and Schult, Daniel A. and Swart, Pieter J.},
    title =        {Exploring Network Structure, Dynamics, and Function using NetworkX},
    booktitle =   {Proceedings of the 7th Python in Science Conference},
    pages =     {11 - 15},
    year =   {2008},
}

@article{freeman1978centrality_conceptual_clarification,
	title = {Centrality in social networks conceptual clarification},
	volume = {1},
	number = {3},
	journal = {Social Networks},
	author = {Freeman, Linton C.},
	year = {1978},
	pages = {215--239},
}

@book{hatcher2002algebraictopology,
    author = {Hatcher, Allen},
    title = {Algebraic Topology},
    publisher = {Cambridge University Press},
    year = {2002}
}

@data{atran2009philippinesambassador,
    author = {S. Atran et al.},
    title = {Philippines Ambassador Residence Bombing 2000, Jakarta.},
    year = {2009},
    publisher = {John Jay & ARTIS Transnational Terrorism Database},
}

@article{mcauley2012discovering_social_circles,
  author       = {Julian McAuley and
                  Jure Leskovec},
  title        = {Discovering Social Circles in Ego Networks},
      year = {2014},
    publisher = {Association for Computing Machinery},
    volume = {8},
    number = {1},
  journal = {ACM Transactions on Knowledge Discovery from Data (TKDD)},

}

@article{knight2011internet_topology_zoo,
  author={Knight, Simon and Nguyen, Hung X. and Falkner, Nickolas and Bowden, Rhys and Roughan, Matthew},
  journal={IEEE Journal on Selected Areas in Communications}, 
  title={The Internet Topology Zoo}, 
  year={2011},
  volume={29},
  number={9},
  pages={1765--1775},
}

@article{grant_1973_dominance,
	title = {Dominance and association among members of a captive and a free-ranging group of grey kangaroos ({Macropus} giganteus)},
	volume = {21},
	number = {3},
	journal = {Animal Behaviour},
	author = {Grant, T. R.},
	year = {1973},
	pages = {449--456},
}

@data{kaminski2018moviegalaxies,
author = {Kaminski, Jermain and Schober, Michael and Albaladejo, Raymond and Zastupailo, Oleksandr and Hidalgo, César},
publisher = {Harvard Dataverse},
title = {Moviegalaxies - Social Networks in Movies},
year = {2018},
version = {V3},
}

@article{read1954cultures_central_highlands,
    author = {K. E. Read},
    title = {Cultures of the Central Highlands, {N}ew {G}uinea},
    journal = {Journal of Anthropological Research},
    year = {1954},
    volume = {10},
    number = {1}
}

@article{rhodes2009inferring_missing_links,
	title = {Inferring missing links in partially observed social networks},
	volume = {60},
	number = {10},
	journal = {Journal of the Operational Research Society},
	publisher = {Taylor \& Francis},
	author = {Rhodes, C. J. and Jones, P.},
	year = {2009},
	pages = {1373--1383},
}

@book{hobson2013evolution_modern_capitalism,
    author = {Hobson, J. A.},
    title = {The Evolution of Modern Capitalism (Routledge Revivals)},
    subtitle = {A Study of Machine Production},
    publisher = {Routledge},
    year = {2013},
}

@article{gelardi2020measuring_social_networks_primates,
	title = {Measuring social networks in primates: wearable sensors versus direct observations},
	volume = {476},
	number = {2236},
	journal = {Proceedings of the Royal Society A: Mathematical, Physical and Engineering Sciences},
	author = {Gelardi, Valeria and Godard, Jeanne and Paleressompoulle, Dany and Claidiere, Nicolas and Barrat, Alain},
	year = {2020},
}

@misc{bunch2021weighting_vectors_boundary_detection,
    author = {Bunch, Eric and Kline, Jeffery and Dickinson, Daniel and Bhat, Suhaas and Fung, Glenn},
    title = {Weighting vectors for machine learning: numerical
harmonic analysis applied to boundary detection},
    year = {2021},
    eprint={2106.00827},
      archivePrefix={arXiv},
}

@book{read1998atlas_of_graphs,
	title = {An Atlas Of Graphs},
	publisher = {Oxford University Press},
	author = {Read, Ronald C and Wilson, Robin J},
	year = {1998},
}

\clearpage

\appendix

\section{Implementation}
\label{app:code}
We briefly give some further explanations on how we compute the (Eulerian) magnitude homology centrality. The main part of the algorithm is to obtain the generators of the respective (Eulerian) magnitude chains. We have implemented a naive algorithm using recursion to achieve this, see \Cref{alg:buildgenerator}. With this part solved, we generate all possible partitions and construct all generators of the (Eulerian) magnitude chains. Then we can obtain the matrix representing the differential and finally obtain the dimension of the homology by computing the ranks of the corresponding differential matrices. We provide the full code on \href{https://github.com/aidos-lab/mh_cent.git}{GitHub}. For our naive implementation, we already obtain a complexity of $\mathcal{O}(n^k)$ for computing all generators with a given starting vertex and a given partition of $l$ into $k$ summands in the standard case and $\mathcal{O}((n\cdot k)^k)$ in the Eulerian case, since we need to check that each vertex appears only once in the generator. 

\begin{algorithm}
    \caption{Given a partition $\mathrm{par}$ of an integer $l\geq 0$ into precisely $k\geq$ summands (with the convention that $\mathrm{par}$ is empty for $k=l=0$), a graph $G=(V,E)$, a vertex $x_0 \in V$, returns all generators of the (Eulerian) magnitude chain group of $G$ that start with vertex $v_0$ and such that the distances between consecutive vertices correspond to the given partition $\mathrm{par}$}
    \label{alg:buildgenerator}
    \begin{algorithmic}
        \Function{BuildGenerator}{$G, x_0, \mathrm{par}, \mathrm{Eulerian}$}
            \State $\mathrm{generators} \gets []$
            \If{$\mathrm{par}$ is empty} \Comment{Treat base cases}
                \State append $(x_0)$ to $\mathrm{generators}$
            \ElsIf{$\mathrm{par}$ contains one element}
                \State $X_1 \gets$ vertices $x_1 \in \V(G)$ with distance $d(x_0, x_1) = \mathrm{par}[0]$
                \For{$x_1$ in $X_1$}
                    \State append $(x_0,x_1)$ to $\mathrm{generators}$
                \EndFor
            \Else \Comment{Recursively build generators}
                \State $X_1 \gets$ vertices $x_1 \in \V(G)$ with distance $d(x_0, x_1) = \mathrm{par}[0]$ 
                \For{$x_1$ in $X_1$}
                    \State $\mathrm{tails} \gets \Call{BuildGenerator}{G, x_1, \mathrm{par}[1:], \mathrm{Eulerian}}$
                    \For{$\mathrm{tail}$ in $\mathrm{tails}$}
                        \If{Eulerian}
                            \If{$x_0 \notin \mathrm{tail}$}
                                \State insert $x_0$ at the beginning of $\mathrm{tail}$ and append the tuple to $\mathrm{generators}$
                            \EndIf
                        \Else
                            \State insert $x_0$ at the beginning of $\mathrm{tail}$ and append $\mathrm{tail}$ to $\mathrm{generators}$
                        \EndIf
                    \EndFor
                \EndFor
            \EndIf
        \State \Return $\mathrm{generators}$
        \EndFunction
        \end{algorithmic}
\end{algorithm}

\end{document}